\documentclass{amsart}
\usepackage{latexsym,amsxtra,amscd,ifthen}
\usepackage{amsfonts}
\usepackage{verbatim}
\usepackage{amsmath}
\usepackage{amsthm}
\usepackage{amssymb}
\theoremstyle{plain}
\newtheorem*{prop}{Proposition}

\newtheorem*{thrm}{Theorem}
\newtheorem*{cor}{Corollary}

\newtheorem{lem}{Lemma}
\newtheorem*{lemma}{Lemma}

  \theoremstyle{remark}

\newtheorem*{exs}{Examples}
\newtheorem*{defn}{Definition}

\newtheorem*{rem}{Remark}
\newtheorem*{remarks}{Remarks}

\newtheorem*{rems}{Remarks}

\newtheorem*{question}{Question}
\theoremstyle{definition}

\DeclareMathOperator{\gr}{gr}

\DeclareMathOperator{\GKdim}{GKdim}

\begin{document}

\title[Connected Hopf algebras and
iterated Ore extensions]
{Connected Hopf algebras and \\
iterated Ore extensions}

\author{K.A. Brown, S. O'Hagan, J.J. Zhang and G. Zhuang}

\address{Brown and O'Hagan: School of Mathematics and Statistics,
University of Glasgow, Glasgow G12 8QW, UK}

\email{Ken.Brown@glasgow.ac.uk}

\email{steven.ohagan@gmail.com}

\address{Zhang: Department of Mathematics,  University of Washington,
Box 354350, Seattle, WA 98195-4350, USA}

\email{zhang@math.washington.edu}

\address{Zhuang: Department of Mathematics, University of Southern
California, 3620 S. Vermont Ave., Los Angeles, CA 90089-2532, USA}

\email{gzhuang@usc.edu}


\begin{abstract} We investigate when a skew polynomial extension
$T = R[x;\sigma,\delta]$ of a Hopf algebra $R$ admits a Hopf algebra
structure, substantially generalising a theorem of Panov. When this
construction is applied iteratively in characteristic 0 one obtains a
large family of connected noetherian Hopf algebras of finite
Gelfand-Kirillov dimension, including for example all enveloping
algebras of finite dimensional solvable Lie algebras and all coordinate
rings of unipotent groups. The properties of these Hopf algebras are
investigated.
\end{abstract}

\maketitle

\section{Introduction}
\label{intro}
\subsection{}
This paper develops lines of research begun in \cite{Zh1} and in \cite{Pa}.
Zhuang in \cite{Zh1} initiated the study of connected affine Hopf
$k$-algebras of finite Gelfand-Kirillov dimension
(or GK-dimension, for short), over an algebraically
closed field $k$ of characteristic 0; that study was continued in
\cite{WZZ}, and is taken further here. In \cite{Pa}, Panov asked the
question: given a field $F$ and a Hopf $F$-algebra $R$, for which
algebra automorphisms $\sigma$ and $\sigma$-derivations $\delta$ can
the skew polynomial algebra $T = R[x;\sigma,\delta]$ be given a
structure of Hopf algebra extending the given structure on $R$?
(See (\ref{Oredefn}) for the definition of skew polynomial algebra.)
Here, we answer Panov's question in some special cases.
In the following two paragraphs we discuss these two themes
in more detail, starting with the second.

\subsection{Hopf Ore extensions}
\label{intro2}
Let $R$, $\sigma$, $\delta$ and $T$ be as in the previous paragraph.
The main theorem of \cite{Pa} answered the above question under the
additional hypothesis that $x$ is a skew primitive element of
$T$ - that is, there are group-like elements $a,b \in R$ such that
$\Delta (x) = a \otimes x + x \otimes b$, where $\Delta$ denotes the
coproduct of $T$. Our main result in this direction is
Theorem \ref{permit}. The following, which appears as part of
Corollary \ref{tensor}, is a consequence of that theorem and of
other lemmas in $\S$\ref{HopfOre}.

\begin{thrm}
Let $R$ be a Hopf $F$-algebra which is a domain. Let
$T = R[x;\sigma,\delta]$ be a skew polynomial algebra over $R$.
\begin{enumerate}
\item
Suppose that $T$ admits a Hopf algebra structure, with $R$ a Hopf
subalgebra. Suppose also that 
\begin{equation}\label{copout}\Delta(x)=a\otimes x+ x\otimes b+v(x\otimes x) +w, \end{equation}
where $a,b\in R$ and $v,w\in R\otimes R$.
Then, {\rm{(}}possibly after a change of the variable
$x${\rm{)}},
\begin{enumerate}
\item[(a)]
$\Delta (x) = a \otimes x + x \otimes 1 + w$, where $a$ is a
group-like element of $R$ and $w = \sum w_1 \otimes w_2 \in R \otimes R$;
\item[(b)]
$\varepsilon (x) = 0$ and $S(x) = -a^{-1}(x + \sum w_1 S(w_2))$;
\item[(c)]
There is a character $\chi:R \longrightarrow F$ such that
$$ \sigma = \tau^{\ell}_{\chi} = \mathrm{ad}(a) \circ \tau^{r}_{\chi};$$
that is, $\sigma$ is a left winding automorphism of $R$, and is the
composition of the corresponding right winding automorphism with
conjugation by $a$;
\item[(d)]
$\sigma, \delta, w$ and $a$ satisfy a specific set of identities.
\end{enumerate}
\item
Conversely, given $\sigma, \delta, w$ and $a$ as in {\rm{(i)}},
$T$ can be given a Hopf algebra structure extending the structure
on $R$, with $\Delta, \varepsilon, S$ as above.
\end{enumerate}
\end{thrm}

The identities referred to in (i)(d) of the theorem are \eqref{delta},
\eqref{antiw} and \eqref{pinned} in Theorem \ref{permit}. We call
a Hopf algebra $T$ which is a skew polynomial algebra over a Hopf
subalgebra $R$, as in part (i) of the above theorem,
a \emph{Hopf Ore extension} (HOE) of $R$ - see Definition \ref{HOEdef}.

\subsection{The coproduct of the indeterminate}\label{coprod} 
A very obvious question immediately arises from Theorem \ref{intro2}: 
Given a Hopf $F$-algebra $T = R[x;\sigma,\delta]$, with $R$ a 
Hopf subalgebra, can we always change the variable $x$ so that 
$\Delta (x)$ has the form (\ref{copout})? We don't know the 
answer to this question. But we do show, in Lemma 1 of 
paragraph \ref{coHOE}, that $\Delta (x)$ does have the form 
(\ref{copout}), provided we allow the first two terms to be 
$a(1 \otimes x)$ and $b(x \otimes 1)$ with $a$ and $b$ in 
$R \otimes R$, and provided we require $R \otimes R$ to be a domain.

A central motivation for us in this work has been the study of \emph{connected} Hopf $k$-algebras (that is, Hopf $k$-algebras whose coradical is $k$), where $k$ is algebraically closed of characteristic 0. For these, we can indeed give a positive answer to the above question. The following result appears as Proposition \ref{skewconn}.

\begin{thrm}
Let $k$ be algebraically closed of characteristic 0, and $R$
a connected Hopf $k$-algebra and let $T = R[x; \sigma, \delta]$ be
a Hopf algebra containing $R$ as a Hopf subalgebra. Then
$$\Delta(x)=1\otimes x+x\otimes 1+w$$
for some $w\in R\otimes R$. Consequently, $T$ satisfies all the conclusions of Theorem \ref{intro2}, with $a = 1$. Moreover, $T$ is connected.
\end{thrm}

\subsection{Iterated Hopf Ore extensions}
\label{IHOEintro}
The preservation of the connected property in Theorem \ref{coprod} makes it very natural to adopt an inductive definition. Thus we define an \emph{iterated Hopf Ore extension} (IHOE)
to be a Hopf $F$-algebra $H$ containing a chain of Hopf subalgebras
\begin{equation}
\label{introdef1}
F = H_{(0)} \subset \cdots \subset H_{(i)} \subset H_{(i+1)}
\subset \cdots \subset H_{(n)} = H,
\end{equation}
with each of the extensions $H_{(i)} \subset H_{(i+1)} :=
H_{(i)}[x_{i+1}; \sigma_{i+1},\delta_{i+1}]$ being a HOE.
From repeated applications of Theorem \ref{coprod}, using also Theorem \ref{GKdim}, we obtain:

\begin{thrm}
If $H$ is an IHOE as in \eqref{introdef1}, 
{\rm{(}}with the base field algebraically closed of characteristic 0{\rm{)}}, 
then $H$ is a connected
noetherian Hopf algebra of GK-dimension $n$. After changes of the
defining variables $\{x_i\}$ but not of the chain \eqref{introdef1},
for each $i = 1, \ldots , n,$
\begin{equation}
\Delta (x_i) =  1 \otimes x_i + x_i \otimes 1 + w^{i-1},
\end{equation}
where $w^{i-1} \in H_{(i-1)} \otimes H_{(i-1)}$, for $i = 1, \ldots , n$,
with $w^0 = w^1 = 0$. Moreover, the data
$\{x_i,\sigma_j, \delta_j, w^{i-1} : 2 \leq j \leq n,\, 1 \leq i \leq n \}$
satisfy the conditions listed in Theorem {\rm{\ref{prowinding}}},
with $a = 1.$
\end{thrm}

\subsection{Classical IHOEs}
\label{classicIHOE}
Theorem \ref{IHOEintro} thus provides a large supply of connected
Hopf algebras. As we now recall, many, but not all, of the classically
familiar connected Hopf algebras are IHOEs. First, consider the case
when $H$ is an affine commutative Hopf $k$-algebra, $k$ as usual algebraically closed of characteristic 0. The arguments and references needed for the following are detailed in
Examples \ref{iterdefn}(i); the non-trivial part is due to Lazard \cite{Laz}.

\begin{thrm}
Let $H$ be the coordinate ring of the affine algebraic group $G$ over the algebraically closed field $k$ of characteristic 0.
Then the following are equivalent.
\begin{enumerate}
  \item $H$ is an IHOE;
  \item $H$ is a polynomial algebra over $k$;
  \item $H$ is a connected Hopf algebra;
  \item $G$ is unipotent.
\end{enumerate}
\end{thrm}

In particular, notice that every commutative polynomial Hopf $k$-algebra
is an IHOE - this is the dual form of the fact that every unipotent
group $G$ over $k$ has a chain of normal subgroups of length $\dim G$
with each link in the chain isomorphic to the additive group of $k$. In
contrast to this, for \emph{cocommutative} Hopf $k$-algebras with the
appropriate finiteness condition, the connected ones are indeed
all ``noncommutative polynomial algebras'', but not all of them are
IHOEs. This is a straightforward consequence of the description of
cocommutative Hopf $k$-algebras in the theorem of
Cartier-Gabriel-Kostant, \cite[Theorem 5.6.5]{Mon}; the following
result follows from it and the discussion in Examples
\ref{iterdefn}(ii),(iii),(iv).

\begin{thrm}
Let $H$ be a cocommutative Hopf $k$-algebra.
\begin{enumerate}
  \item
$H$ is connected if and only if it is isomorphic as a Hopf algebra
to the enveloping algebra $U(\mathfrak{g})$ of a Lie $k$-algebra
$\mathfrak{g}$.
  \item
Let $H$ be as in {\rm{(i)}}. Then $H$ has finite GK-dimension if and
only if $\mathfrak{g}$ has finite dimension, and in this case
$\mathrm{GKdim} H = \mathrm{dim} \; \mathfrak{g}$.
  \item
Let $H =  U(\mathfrak{g})$ with $\mathrm{dim}\; \mathfrak{g}$ finite.
Then $H$ is an IHOE if and only if the only nonabelian simple factor
occurring in $\mathfrak{g}$ is $\mathfrak{sl}(2,k)$.
\end{enumerate}
\end{thrm}

\subsection{Properties of IHOEs}
\label{propintro}
In view of the theorems in (\ref{classicIHOE}) one expects IHOEs to share
many properties in common with coordinate rings of unipotent groups and
with enveloping algebras of solvable Lie algebras. Indeed this is the
case even for all connected Hopf $k$-algebras of finite GK-dimension:
in the central result of \cite{Zh1} it was shown that, for such an
algebra $A$, the associated graded algebra $\mathrm{gr}A$ with respect
to the coradical filtration is a commutative polynomial Hopf algebra
in $n := \mathrm{GKdim} A$ variables, and hence in particular
$\mathrm{gr} A$ is the coordinate ring of a unipotent group of
dimension $n$. Here we build on this result in Theorem \ref{firstprop};
selected parts of that result, additional to those already given in
Theorem \ref{IHOEintro}, are the following.

\begin{thrm} Let $H$ be an IHOE as in {\rm{(\ref{introdef1})}}.
\begin{enumerate}
\item
After a suitable change of variables {\rm{(}}with Theorem
{\rm{\ref{intro2}(i)}} still valid{\rm{)}}, the Lie algebra
$P(H)$ of primitive elements of $H$ is a subspace of $\sum_{i=1}^n kx_i$.
\item
$H$ is Auslander-regular and AS-regular of
dimension $n$, and is GK-Cohen-Macaulay.
\item
$H$ has Krull dimension at most $n$.
\item
$H$ is skew Calabi-Yau with Nakayama
automorphism $\nu$, where, for $i = 1, \ldots , n$,
$$ \nu (x_i) = x_i + a_i,$$
with $a_i \in H_{(i-1)}.$ If $x_i \in P(H)$ then $a_i \in k.$
\item
The antipode $S$ of $H$ either has infinite order, or $S^2 = \mathrm{Id}$.
\item
$S^4 = \tau^{\ell}_{\chi}\circ \tau^r_{-\chi}$, the composition of a
left winding automorphism of $H$ with the right winding automorphism
of the inverse character, with the character $\chi$ in the centre of
the character group $X(H)$. In particular, $S^4$ is a unipotent
automorphism of a generating finite dimensional subcoalgebra of $H$.
\end{enumerate}
\end{thrm}

Moreover, in Theorem \ref{PI}, partially generalising a well-known
property of enveloping algebras in characteristic 0, we show that

\begin{thrm}
An IHOE  $H$ as in {\rm{(\ref{introdef1})}} which satisfies a
polynomial identity is commutative.
\end{thrm}

\subsection{The maximal classical subgroup of an IHOE}
\label{maxclassic}
As is well-known, for any Hopf algebra $H$ the ideal
$I(H) := \langle [H,H]\rangle$ of $H$ generated by the commutators is
a Hopf ideal. Suppose that $H$ is an affine $k$-algebra with $k$
algebraically closed of characteristic 0. Then $H/I(H)$ is the coordinate
ring of an algebraic group over $k$. It is natural to call this group
\emph{the maximal classical subgroup of} $H$; of course it is nothing
else but the character group $X(H)$ of all algebra homomorphisms from
$H$ to $k$. That is,
$$ H/I(H) \cong \mathcal{O}(X(H)).$$
Suppose now that $H$ is an IHOE with chain (\ref{introdef1})
(although the following remarks are valid more generally for \emph{any}
connected Hopf $k$-algebra of finite GK-dimension $n$). Since Hopf
algebra factors of connected Hopf algebras are again connected, by
\cite[Corollary 5.3.5]{Mon}, it follows from Theorem \ref{IHOEintro}
that $H/I(H)$ is connected. Hence, by the first theorem in
$\S$\ref{classicIHOE}, $H/I(H)$ is a polynomial algebra in at most
$n$ variables; equivalently, $X(H)$ is a unipotent group of dimension
at most $n$.

\subsection{Generalisation of part of a theorem of Goodearl}
\label{geegee}
To study $X(H)$ when $H$ is an IHOE, in $\S$4 (most of which is
independent of the rest of the paper), we prove the following result,
which generalises a special case of a result of Goodearl \cite{Go}
on the prime spectrum of $R[x;\sigma,\delta]$ when $R$ is commutative.
Note that the result below concerns an arbitrary skew polynomial
$F$-algebra, not necessarily a Hopf algebra, and the algebraically
closed field $F$ may have any characteristic.

\begin{thrm}
Let $F$ be algebraically closed and let $R$ be a $F$-algebra. Let
$\sigma$ and $\delta$ be respectively an $F$-algebra automorphism
and a $\sigma$-derivation of $R$, and set $T = R[x; \sigma, \delta]$.
Let $\Xi (T)$ denote the set of ideals $M$ of $T$ with $T/M \cong F$,
and similarly for $\Xi (R)$. Write
$\Psi: \Xi (T) \longrightarrow \Xi (R): M \mapsto M \cap R.$
\begin{enumerate}
\item
Let $M \in \Xi(T)$ and denote $\Psi(M)$ by $\mathfrak{m}.$ Then
{\rm{(a)}} $\delta( [R,R]) \subseteq \mathfrak{m}$, and
either {\rm{(b)}} $\mathfrak{m}$ is $(\sigma, \delta)$-invariant, or
{\rm{(c)}} $\mathfrak{m}$ is not $\sigma$-invariant.
\item
Let $\mathfrak{m} \in \Xi (R)$, and suppose that {\rm{(b)}}
{\rm{(}}and hence {\rm{(a))}} hold for $\mathfrak{m}$. Then
$\mathfrak{m}T \lhd T$ and $T/\mathfrak{m}T \cong F[x]$, so that
$$ \Psi^{-1} (\mathfrak{m}) = \{ \langle \mathfrak{m}T,
x - \lambda \rangle : \lambda \in F \} \cong \mathbb{A}^{1}_F.$$
\item
Let $\mathfrak{m} \in \Xi (R)$, and suppose that {\rm{(a)}} and
{\rm{(c)}} hold for $\mathfrak{m}$. Then there exists a unique
$M \in \Xi (T)$ with $\Psi (M) = \mathfrak{m}.$
\end{enumerate}
\end{thrm}

\subsection{Applications of Theorem \ref{geegee} to IHOEs}
\label{app}
There are several reasons why $X(H)$ is a worthwhile object of
study for an affine Hopf algebra $H$:

$\indent \bullet$ it is an important invariant of $H$;

$\indent \bullet$ it yields (via the left and right winding
automorphisms) a supply of algebra automorphisms of $H$;

$\indent \bullet$ it is a starting point for the representation
theory of $H$;

$\indent\bullet$ the behaviour of $X(H)$ under restriction and
induction has structural consequences for $H$.

In this paper we focus on the last of these points. To explain what
we mean by it, let $T = R[x; \sigma,\delta]$ be a HOE; as in
Theorem \ref{geegee}, let $\Psi$ be the restriction map, which we
view as a homomorphism of groups from $X(T)$ to $X(R)$. Let
$\mathfrak{m} \in \mathrm{im}\Psi$, a closed subgroup of $X(R)$.
Theorem \ref{geegee} tells us that $\Psi^{-1}(\mathfrak{m})$ is
either a single element of $X(T)$ or is a copy of $(k,+)$. If
$\tau$ is any non-identity winding automorphism of $T$ then
$\tau$ permutes the characters of $T$ transitively, with $\tau(R)=R$
since $R$ is a Hopf subalgebra of $T$, we see that the fibre
$\Psi^{-1}(\mathfrak{m})$ takes the same form for \emph{every}
$\mathfrak{m}$ in $\mathrm{im}\Psi$. In the first case we say
that the HOE is of \emph{variant type}, and in the second we call
it \emph{invariant type}. Note that the extension is of invariant
type if and only if $R^+T$ is an ideal of $T$, where $R^+$ is the
augmentation ideal of $R$. This is all contained in $\S$4. In $\S$5,
we take the first steps in exploring this dichotomy: the following
result summarises parts of Theorems \ref{invHOEs} and \ref{varHOEs}.

\begin{thrm}
Let $k$ be an algebraically closed field of characteristic 0, let
$R$ be a Hopf $k$-algebra and let $T =R[x; \sigma, \delta]$ be a HOE.
\begin{enumerate}
\item
If $T$ is an invariant HOE of $R$ then there is a change of variables
from $x$ to $\tilde{x}$ so that $T = R[\tilde{x}; \partial]$, where
$\partial$ is a derivation of $R$.
\item
If $R$ is an affine commutative domain and $T$ is a variant HOE,
then there is a change of variables from $x$ to $\tilde{x}$ such
that $\tilde{x}$ is skew primitive and $T = R[\tilde{x}; \sigma]$.
\end{enumerate}
\end{thrm}

\subsection{Notation} \label{notation}
Throughout, $F$ is an arbitrary field, and $k$ is an algebraically
closed field of characteristic $0$. Given a Hopf algebra $H$, whether
over $F$ or over $k$, we shall use the standard symbols
$\Delta, S, \varepsilon$ for the coproduct, antipode and counit of
$H$, and, for $h \in H$, write $\Delta (h) = h_1 \otimes h_2$. We
denote the augmentation ideal of $H$ by $H^+$, and write $G(H)$
for the set of group-like elements of $H$.  The
\emph{coradical filtration} of $H$, as defined at \cite[\S 5.2]{Mon},
is denoted by $\{H_n : n \geq 0 \}.$ We shall always assume that
the antipode is bijective, although this is not a significant
imposition in the present paper, since it is known to be always
true for semiprime right noetherian Hopf algebras by \cite{Sk}.

\section{Hopf Ore extensions}
\label{HopfOre}
\subsection{Definition}\label{Oredefn}
Recall that, if $R$ is a $F$-algebra and $\sigma$ is a $F$-algebra
automorphism of $R$, then a \emph{left $\sigma$-derivation $\delta$}
of $R$ is a $F$-linear endomorphism of $R$ such that
$\delta(ab) = \delta(a)b + \sigma(a)\delta(b)$ for all $a,b \in R$.
Given these ingredients, the \emph{skew polynomial algebra}
$T = R[x; \sigma, \delta]$ is the $F$-algebra generated by $R$ and
$x$, subject to the relations
\begin{equation} \label{relations}
x r - \sigma(r) x = \delta(r)
\end{equation}
for all $r \in R$. For basic properties, see for example \cite[I.1.11]{BG}.

\begin{defn} \label{HOEdef} Let $R$ be a Hopf $F$-algebra.
A \emph{Hopf Ore extension} (HOE) of $R$ is a $F$-algebra $T$ such
that
\begin{itemize}
\item
$T$ is a Hopf $F$-algebra with Hopf subalgebra $R$;
\item
there exist an algebra automorphism $\sigma$ and a
$\sigma$-derivation $\delta$ of $R$ such that $T = R[x;\sigma,
\delta]$.
\item
there are $a,b\in R$ and $v,w\in R\otimes R$ such that
\begin{equation}
\label{defnHOE3}
\Delta(x)=a\otimes x+x\otimes b+v(x\otimes x)+w.
\end{equation}
\end{itemize}
\end{defn}

We repeat here in a more precise form the question already raised in paragraph \ref{coprod}:

\begin{question} Does the third condition in the definition of a HOE follow from the first two, after changing the variable $x$?
\end{question}

We show in Lemma 1 of the next subsection that one can get some way towards a positive answer with a relatively weak additional hypothesis on $R$, namely that $R \otimes R$ is a domain; and in Proposition \ref{skewconn} we shall completely answer the question when $R$ is a connected $k$-algebra. But otherwise, the question is open.

\subsection{Comultiplication in HOEs}\label{coHOE}
Our initial target in $\S 2$ is Theorem \ref{prowinding}, where we
determine constraints on the possible HOEs which can be formed with
coefficient ring a given Hopf $F$-algebra $R$. Such a result has
previously been obtained by Panov \cite{Pa}, under the additional
hypothesis that the variable $x$ of the extension is skew primitive.
Typically, this hypothesis is not valid: consider for example the
coordinate algebra $H = \mathcal{O}(G)$ of the Heisenberg group
$G$ of dimension 3. View $G$ as the set of upper triangular
$3 \times 3$ matrices with 1s on the main diagonal, so $H$ is
generated by the coordinate functions $y,z$ and $x$ for entries
$(1,2), (2,3)$ and $(1,3)$ respectively. Then
$H = k[x,y,z] = k[y,z][x]$, a HOE with coefficient Hopf algebra
$R = k[y,z]$, but $x$ is not skew primitive:
$$ \Delta (x) = x \otimes 1 + 1 \otimes x + y \otimes z. $$
Indeed the space of primitive elements of $H$ is spanned by $y$ and
$z$, so no alternative choice of third variable renders $H$ as a
HOE of $R$ with skew primitive generator.

\begin{lem}\label{comulti}
Let $T=R[x; \sigma, \delta]$ be a Hopf $F$-algebra with $R$ a Hopf
subalgebra. Suppose that $R\otimes R$ is a domain. Then
  $\Delta(x)=s(1\otimes x)+t(x\otimes 1)+v(x\otimes x)+w$,
where $s,t, v$ and $w\in R\otimes R$.
\end{lem}

\begin{proof}
  Since $T$ is a free left $R$-module on basis $\{x^i: i\ge 0\}$, we can write
\begin{equation}\label{eq1}
\Delta(x)=\sum_{i, j\ge 0}w_{ij}x^i\otimes x^j
\end{equation}
where each $w_{ij} \in R\otimes R$. Fix $j_0$ to be the maximal integer such
that $w_{ij_0}\neq 0$ for some $i$ and let $i_0$ be the maximal
integer such that $w_{i_0 j_0}\neq 0$. Define $\deg r=0$ for all
$r\in R$ and $\deg x=1$. Then we can extend this to define the
lexicographic order,
which we shall call {\it degree}, for the elements of $T\otimes T$ and
$T\otimes T\otimes T$.

Applying  $\Delta\otimes 1$ to $(\ref{eq1})$ we get
\begin{equation}\label{eq2}
(\Delta\otimes 1)\Delta(x)=\sum_{i, j}(\Delta\otimes 1)
(w_{ij})((\Delta x)^i\otimes x^j)
\end{equation}
with maximal degree component $\gamma x^{i_0^2}\otimes x^{j_0 i_0}\otimes
x^{j_0}$ where $$ \gamma =(\Delta\otimes 1)(w_{i_0j_0}) (w_{i_0j_0}
(\sigma^{i_0}\otimes \sigma^{j_0})(w_{i_0 j_0})\cdots
(\sigma^{i_0}\otimes \sigma^{j_0})^{i_0-1}(w_{i_0 j_0})\otimes 1).$$
Since $R\otimes R$ is a domain, $ \gamma \neq 0$.

Similarly,
\begin{align}\label{eq3}
(1\otimes \Delta)\Delta(x)= &
\sum_{i, j}(1\otimes \Delta)(w_{ij})(x^i\otimes (\Delta x)^j)\\
\nonumber =&
\sum_{i, j}(1\otimes \Delta)(w_{ij})
(x^i\otimes (\sum_{s, t}w_{st}x^s\otimes x^t)^j)
\end{align}
with maximal degree component $\nu x^{i_0}\otimes x^{i_0 j_0}\otimes
x^{j_0^2}$ where
$$ \nu =(1\otimes \Delta)(w_{i_0j_0}) (1\otimes w_{i_0j_0}
(\sigma^{i_0}\otimes \sigma^{j_0})(w_{i_0 j_0})\cdots
(\sigma^{i_0}\otimes \sigma^{j_0})^{j_0-1}(w_{i_0 j_0})).$$ Since
$R\otimes R$ is a domain, $\nu \neq 0$.

Since (\ref{eq2}) equals (\ref{eq3}) by coassociativity, we must have
$j_0=0$ or $1$. By symmetry, the maximal value of $i$ such that
$w_{ij}\neq 0$ for some $j$ is $i=0$ or $i=1$. Hence there exists
$s,t, v$, and $w\in R\otimes R$ such that
\begin{equation}\label{eq4}
\Delta(x)=s(1\otimes x)+t(x\otimes 1) +v(x\otimes x)+ w.
\end{equation}
\end{proof}

 We repeat here for emphasis that, as a special case of the question in subsection \ref{Oredefn}, we can ask: In the above lemma, can the conclusion (\ref{eq4}) be replaced by (\ref{defnHOE3})? In partially answering this question over the next three pages we will continue to use the notation of Definition \ref{Oredefn}.

\begin{lem}
\label{gplike}
Let $T=R[x; \sigma, \delta]$ be a HOE. 
\begin{enumerate}
\item
$a,b \in G(R)$ provided that $v = 0$ or $w = 0$.
\item
$v \in F.$ \item If $v = 0$ then
\begin{equation} \label{wcond}
a \otimes w + (\mathrm{Id} \otimes
\Delta)(w) = w \otimes b + (\Delta \otimes \mathrm{Id})(w).
\end{equation}
\end{enumerate}
\end{lem}

\begin{proof} (i) Applying $\Delta\otimes 1$ and $1\otimes \Delta$,
respectively, to \eqref{defnHOE3}, we have
\begin{align}\label{eq5}
(\Delta\otimes 1)&\Delta(x)\\
\nonumber&= \Delta a \otimes x+(a\otimes x+x\otimes b+v(x\otimes x)+w)
\otimes b\\
\nonumber&\quad +((\Delta\otimes 1)v)((a\otimes x+x\otimes b+v(x\otimes
x)+w)\otimes x)+(\Delta\otimes 1)w,
\end{align}
and
\begin{align}\label{eq6}
(1\otimes \Delta)&\Delta x\\
\nonumber &= a\otimes (a\otimes x+x\otimes b+v(x\otimes x)+w)
+ x\otimes \Delta b\\
\nonumber &\quad + ((1\otimes \Delta)v)(x\otimes
(a\otimes x+x\otimes b+v(x\otimes x)+w))+(1\otimes \Delta)w
\end{align}
Equating coefficients in $R \otimes R \otimes R$ of $1\otimes 1\otimes x$
in $(\ref{eq5})$ and $(\ref{eq6})$ yields
\begin{equation}\label{eq7}
\Delta(a)\otimes 1+((\Delta\otimes 1)v)(w\otimes 1)=a\otimes a\otimes 1.
\end{equation}
Consider coefficients of $x\otimes 1\otimes 1$ in $(\ref{eq5})$ and
$\ref{eq6}$ to get
\begin{equation}\label{eq8}
1\otimes b\otimes b=1\otimes \Delta b+((1\otimes \Delta)v)(1\otimes w).
\end{equation}
Suppose that
\begin{equation}
\text{$w=0$ or $v=0$.}
\end{equation}
Then $(\ref{eq7})$ and $(\ref{eq8})$ at once yield (i).

(ii) \textbf{Case 1:} $w \neq 0.$ Compare the right-most tensorands
in \eqref{eq7} to get $v = v_1 \otimes 1.$ A similar comparison of
the left-most tensorands in \eqref{eq8} yields $v = 1 \otimes v_2.$
It follows that $v \in F.$

\textbf{Case 2:} $w = 0.$ Comparing the coefficients of
$1 \otimes x \otimes x$ in (\ref{eq5}) and (\ref{eq6}) shows that
\begin{equation} \label{wzero}
((\Delta \otimes \mathrm{Id})(v))(a \otimes 1 \otimes 1)
= a \otimes v.
\end{equation}
Write $v = \sum c_i \otimes d_i$ with $\{d_i\}$ linearly independent.
Then (\ref{wzero}) shows that, for each $i$,
$$ \Delta(c_i)(a \otimes 1) = (1 \otimes c_i)(a \otimes 1).$$
From (i), given that $w = 0$ we know that $a \in G(R)$, so we deduce
from the above equation that $\Delta (c_i) = 1 \otimes c_i$ for all
$i$. Applying $1 \otimes \varepsilon$ to $\Delta c_i$ we see that
$c_i \in k$ for all $i$, so that $v = 1 \otimes d$ for some
$d \in R.$ Now consider the coefficient of $x \otimes x \otimes 1$
in (\ref{eq5}) and (\ref{eq6}): we find that
$$ v \otimes b = ((\mathrm{Id} \otimes \Delta)(v))(1 \otimes 1 \otimes b). $$
However, this implies that
$$ \Delta (d) = d \otimes 1, $$
and applying $ \varepsilon \otimes \mathrm{Id}$ gives $d \in k$ as well.
Therefore $v \in F$ in Case 2 also.

Finally, (iii) follows by comparing the coefficients of $1 \otimes
1 \otimes 1$ in (\ref{eq5}) and (\ref{eq6}).
\end{proof}

\subsection{The antipode}
\label{anti}

\begin{lemma}\label{Sx}
Let $T=R[x; \sigma, \delta]$ be a Hopf $F$-algebra with $R$ a Hopf
subalgebra. Suppose that $\Delta(x)$ has the form \eqref{defnHOE3},
writing $w=\sum w_1\otimes w_2.$ Then we have the following.
\begin{enumerate}
\item
Suppose that $R$ is a domain. Then $S(x)=\alpha x+ \beta$
for $\beta\in R$ and $\alpha \in R^{\times}$.
\end{enumerate}
For the remainder of the lemma, assume that $S$ has the form obtained in
{\rm{(i)}}.
\begin{enumerate}
\item[(ii)]
$v=0$.
\item[(iii)]
$a, b\in G(R)$.
\item[(iv)]
$\alpha=-a^{-1}\sigma(b^{-1})$.
\item[(v)]
$\beta=a^{-1}(\varepsilon (x)- \delta(b^{-1}) - \sum w_1S(w_2))$.
\end{enumerate}
\end{lemma}

\begin{proof}(i) Since $S$ is an anti-automorphism of $T$ it is
easily checked that, for all $r \in R$,
$$ S(x)r - S\sigma^{-1}S^{-1}(r)S(x) = -S\delta\sigma^{-1}S^{-1}(r).$$
Thus $T$ is an Ore extension of $R$ with indeterminate $S(x)$,
$$ T = R[S(x) ; S\sigma^{-1}S^{-1}, -S\delta\sigma^{-1}S^{-1}].$$
Since $R$ is a domain, (i) follows by considering the expression
for $x$ as a polynomial in $S(x)$ with coefficients from $R$.

(ii) By the defining property of the antipode, and using Lemma
\ref{gplike}(ii) of $\S$\ref{coHOE} at the third equality,
\begin{align}\label{antipode1}
\varepsilon(x) =&\,\mu(1\otimes S)\Delta(x)\\
\nonumber =&\,\mu(1\otimes S)(a\otimes x+x\otimes b+v(x\otimes x)+w)\\
\nonumber=&\, \mu(a\otimes \alpha x +a\otimes \beta +x\otimes Sb+
  v(x\otimes \alpha x+ x\otimes \beta)+\sum w_1\otimes Sw_2)\\
\nonumber=&\, a\alpha x+v x\beta+xSb+vx\alpha x+ a\beta+ \sum w_1Sw_2.
\end{align}
The only term of degree $2$ in $x$ on the right-hand side of
$(\ref{antipode1})$ is
$$v\sigma(\alpha)x^2.$$
This must therefore be zero. By (i), $\sigma(\alpha)$ is a unit of $R$, so
(ii) follows.

(iii) is immediate from (ii) and Lemma \ref{gplike}(i) of $\S$\ref{coHOE}.

(iv),(v) Given (ii), $(\ref{antipode1})$ simplifies to
\begin{align*}
\varepsilon (x)= &\, a\alpha x+xSb+ a\beta+ \sum w_1Sw_2\\
= &\, (a\alpha +\sigma(b^{-1}))x+\delta(b^{-1})+ a\beta+ \sum
w_1Sw_2.
\end{align*}
Using (iii) and equating the coefficient of $x$ and the constant
term, (iv) and (v) follow.
\end{proof}

\subsection{Panov's theorem generalised}\label{permit}
The polynomial variable of a skew polynomial extension is far from
uniquely determined: for if $T = R[x; \sigma, \delta]$ is a skew
polynomial algebra and $\lambda\in F$, then $\delta_{\lambda}:=
\delta+\lambda(Id-\sigma)$ is another $\sigma$-derivation of $R$,
as is easily checked, and
$$T=R[x+\lambda; \sigma, \delta_{\lambda}].$$
Whenever convenient, therefore, we may assume without loss of
generality when studying a HOE $T = R[x;\sigma,\delta]$ that
$$x\in T^+$$
Moreover, given a unit of $R$, say for example $b^{-1}$, replacing
$x$ by $b^{-1}x$, and writing $\mathrm{ad}(b^{-1})$ to denote
conjugation by $b^{-1}$, one easily checks that
$$ T = R[b^{-1}x; \mathrm{ad}(b^{-1})\sigma, b^{-1}\delta].$$
In practice, $b$ will be a group-like element of a Hopf algebra
when we apply this below, so this usage of the notation
$\mathrm{ad}$ coincides with the standard Hopf notation
$\mathrm{ad}_{\ell}$, \cite[page 33]{Mon}.

\begin{thrm}
\label{prowinding}
\begin{enumerate}
\item
Let $R$ be a Hopf $F$-algebra and let $T = R[x; \sigma, \delta]$
be a HOE of $R$. Suppose that
\begin{equation}
\label{(H)}\tag{H}
S(x) = \alpha x + \beta \textit{  for } \alpha,\beta
\in R \textit{ with } \alpha \in R^{\times}.
\end{equation}
Write $w = \sum w_1 \otimes
w_2 \in R \otimes R,$ with $\{w_1\}$  and $\{w_2\}$  chosen to be
$F$-linearly independent subsets of $R.$ Then the following hold.
\begin{enumerate}
\item[(a)]
$a, b \in G(R)$ and $v = 0.$
\item[(b)]
After a change
of the variable $x$ and corresponding adjustments to $\sigma$,
$\delta$ and $w$,
\begin{equation} \label{xcounit} \varepsilon (x) = 0
\end{equation} and
\begin{equation}
\label{copx} \Delta(x)=a\otimes x+ x\otimes 1+ w.
\end{equation}
For the remainder of {\rm{(i)}}, we assume that \eqref{xcounit} and
\eqref{copx} hold.
\item[(c)]
$S(x) = -a^{-1}(x + \sum w_1S(w_2)).$
  \item[(d)]
There is a character $\chi:
R\rightarrow F$ such that
\begin{equation}\label{winding}
\sigma(r)=\chi(r_1)r_2 = ar_1a^{-1}\chi(r_2) = \mathrm{ad}(a)\circ
\tau^r_{\chi}(r),
\end{equation}
for all $r\in R$. That is, $\sigma$ is a left winding automorphism
$\tau^{\ell}_{\chi}$, and is the composition of the corresponding
right winding automorphism with conjugation by $a$.
  \item[(e)]
The $\sigma$-derivation $\delta$ satisfies the relation
\begin{equation} \label{delta}
\Delta\delta(r)-\delta(r_1)\otimes r_2 -ar_1\otimes
\delta(r_2)=w\Delta(r)-\Delta\sigma(r)w.
\end{equation}
\item[(f)]
  The element $w$ is in $R^+ \otimes R^+$, and satisfies the identities
\begin{equation} \label{antiw} S(w_1)w_2 = a^{-1}w_1S(w_2),
\end{equation}
and
\begin{equation} \label{pinned}
w \otimes 1 + (\Delta \otimes \mathrm{Id})(w) = a
\otimes w + (\mathrm{Id} \otimes \Delta)(w).
\end{equation}
\end{enumerate}
\item
Let $R$ be a Hopf $F$-algebra. Suppose given $a \in G(R)$,
$w \in R \otimes R$, a $F$-algebra automorphism $\sigma$ of $R$
and a $\sigma$-derivation $\delta$ of $R$ such that this data
satisfies {\rm{(d), (e)}} and {\rm{(f)}} of {\rm{(i)}}. Then
the skew polynomial algebra $T = R[x; \sigma, \delta]$ admits a
structure of Hopf algebra with $R$ as a Hopf subalgebra, and
with $x$ satisfying $\mathrm{(H)}$, {\rm{(a), (b)}} and
{\rm{(c)}} of {\rm{(i)}}. As a consequence, $T$ is a HOE
of $R$.
\end{enumerate}
\end{thrm}

\begin{proof} (i)(a) Lemma \ref{anti} (ii) and (iii).

(b) This follows from the discussion before the theorem.

(c) Given (b), this is Lemma \ref{anti}(iv) and (v).

(d), (e) The proofs of these are very similar to the corresponding
parts of \cite[Theorem 1.3]{Pa}. In brief, to obtain the first
equality in (d), one first writes down the conditions on $\Delta$
which ensure that the defining relations (\ref{relations}) for $T$
are satisfied. Thus, define
$$\chi:R \longrightarrow R: r \mapsto \sigma (r_1)S(r_2).$$
Calculation using (\ref{relations}) confirms that, for $r \in R$,
$\Delta \circ \chi (r) =
\chi (r) \otimes 1,$ whence $\chi (r) \in k.$ Then one checks that
$\chi$ is a character of $F$, and that $\sigma =
\tau^{\ell}_{\chi}$ as in (d).

Now let $r \in R$ and apply $\Delta$ to the equation $xr =
\sigma(r)x + \delta (r).$ Comparing the coefficients of $x \otimes
1$, of $1 \otimes x$, and of $1 \otimes 1$ yields the rest of (d)
and (e).

(f) (\ref{pinned}) is Lemma \ref{gplike}(ii). From the identities
$(\mathrm{Id} \otimes \varepsilon) \circ \Delta (x) = x =
(\varepsilon \otimes \mathrm{Id}) \circ \Delta (x)$ and the linear
independence of $\{w_1\}$ and $\{w_2\}$ we deduce that the $w_i$
are in the augmentation ideal of $R$. Finally, (\ref{antiw})
follows from the identity $\mu \circ (S \otimes \mathrm{Id}) \circ
\Delta (x) = \varepsilon (x)$, together with (\ref{xcounit}) and
(\ref{copx}).

(ii) We leave to the reader the task of checking that, given the
stated conditions, $\Delta : R \longrightarrow R \otimes R$ can be
extended to an algebra homomorphism from $T$ to $T \otimes T$ by
defining $\Delta (x)$ as in (\ref{copx}). The antipode is defined
as in (c), and the counit $\varepsilon$ is extended to $T$ by
setting $\varepsilon (x) = 0.$ It is a routine but quite long
series of calculations, similar to, but more involved than, those
in the proof of \cite[Theorem 1.3]{Pa}, to show that these
extensions can be achieved and that the Hopf axioms are all
satisfied. To show that $\varepsilon$ is an algebra homomorphism
one needs to confirm that $\varepsilon$ respects the relation
$xr - \sigma (r) x = \delta (r)$ for all $r \in R$; in other words,
one must show that, for all $r \in R$, $\varepsilon \delta(r) = 0.$
This is done by applying $\mu \circ (\mathrm{Id} \otimes \varepsilon)$
to (\ref{delta}). This gives, for each $r \in R$,
$$ 0 = (w_1 \varepsilon(w_2))r - \sigma (r)(w_1 \varepsilon (w_2))
+ ar_1 \varepsilon \delta (r_2).$$
Since $w \in R^+ \otimes R^+$ by (i)(f), it follows that
$$ r_1 \varepsilon \delta (r_2) = 0.$$
Now apply $\varepsilon$ to both sides and use the fact that
$\varepsilon \delta$ is a linear map.
\end{proof}

\subsection{The coradical of a HOE}\label{coradHOE}

\begin{prop}\label{coradical}
Let $R$ be a Hopf $F$-algebra and let $T = R[x; \sigma, \delta]$
be a HOE of $R$. Assume $v=0$ and write $w = \sum w_1 \otimes
w_2 \in R \otimes R$, with $\{w_1\}$ and $\{w_2\}$ chosen to be
$F$-linearly independent subsets of $R$.
Then $R_0=T_0$;
that is, the coradical of $T$ is the coradical of $R$.
\end{prop}

\begin{proof} Since $R$ is a subcoalgebra of $T$, we have $R_0\subset T_0$.
Let $F_nT$ be the $F$-subspace of $T$ spanned by elements of degree
$\le n$ in $x$. Obviously $\{F_nT\}_{n\ge 0}$ is an exhaustive
filtration on $T$ and $F_0T=R$. Since $\Delta(x)$ has the form
\eqref{defnHOE3} with $v = 0$, it is easy to check
$\{F_nT\}_{n\ge 0}$ is a coalgebra filtration on $T$ in the sense
that
$$\Delta(F_nT)\subseteq \sum_{i=0}^{n}F_iT\otimes F_{n-i}T.$$
By \cite[5.3.4 Lemma]{Mon}, $T_0\subseteq F_0T=R$. Hence $T_0 =
R_0$ by \cite[Lemma 5.2.12]{Mon}.
\end{proof}

\subsection{The Gelfand-Kirillov dimension of a HOE}
\label{GKdim}

\begin{thrm} Let $R$ be a Hopf $F$-algebra and let
$T = R[x; \sigma, \delta]$ be a HOE of $R$ for which hypothesis
$\mathrm{(H)}$ of Theorem {\rm{\ref{prowinding}}} holds.
Suppose also that $R$ is a finitely generated $F$-algebra. Then
$$\GKdim T=\GKdim R+1.$$
\end{thrm}

\begin{proof} Note first that, since $\sigma$ is a winding
automorphism of $R$ by Theorem \ref{prowinding}(i)(d), it is
locally finite dimensional, in the sense that every finite
dimensional subspace of $R$ is contained
in a finite dimensional $\sigma$-stable subspace. This is because
each finite subset of $R$ is contained in a finite dimensional
subcoalgebra $C$ of $R$ by \cite[Theorem 5.1.1]{Mon}, and $\sigma
(C) \subseteq C$ from the definition of a winding automorphism. The
result now follows from \cite[Lemma 2.2]{HK}.
\end{proof}

\subsection{HOEs with $R$ a domain}
\label{tensor}

If $R$ is any Hopf $F$-algebra, note that Theorem \ref{prowinding}(ii)
holds for $R$. Whether the converse is true in this generality is
not known, but we have the following result, which is an immediate
corollary of the previous results of this section, in view of Lemma \ref{anti}.

\begin{cor}
Let $R$ be a Hopf $F$-algebra and a domain,
and let $T = R[x; \sigma, \delta]$ be a HOE of $R$. Then the
conclusions of Theorem {\rm{\ref{prowinding}(i)}}, Proposition
{\rm{\ref{coradical}}}, and, if $R$ is affine, Theorem
{\rm{\ref{GKdim}}}, are valid for $T$.
\end{cor}

It is natural to ask:

\begin{question}
Are the results of the above corollary valid without the hypothesis
that $R$ is a domain? Equivalently, is the domain
hypothesis needed for Lemma \ref{anti}(i)?
\end{question}

While we don't know the answer to the above question, we can at least
deal with the case of particular interest to us in the present paper,
namely the case where $R$ is connected of finite GK-dimension, as we
explain in the next subsection. In view of Lemma 1 of subsection \ref{coHOE}, there is a related question:

\begin{question} If a Hopf $F$-algebra $R$ is a domain, is
$R \otimes R$ also a domain?
\end{question}

While we again don't know the answer to this, we can at least note that
if the answer is positive then this will require the Hopf hypothesis.
For in \cite{RS} Rowen and Saltman exhibit division $k$-algebras $E$
and $F$, finite dimensional over their centres, each centre containing
the algebraically closed field $k$ of characteristic 0, with
$E \otimes_k F$ not a domain; and by a famous result of Cohn
\cite[Corollary 6.1]{Co}, $E$ and $F$ can be embedded in a common
division ring $D$ with $k \subseteq Z(D)$.

\subsection{HOEs of connected Hopf algebras}
\label{skewconn}
If $R$ is a connected Hopf $k$-algebra, then so is $R \otimes R$:
for, if $\{R_i\}$ is the coradical filtration of $R$, then it is
clear from the definition, \cite[Theorem 5.2.2]{Mon}, that $A_n :=
\sum_{i = 0}^n R_i \otimes R_{n-i}$ is a coalgebra filtration of
$R \otimes R$, and hence by \cite[Lemma 5.3.4]{Mon} the coradical
of $R \otimes R$ is contained in $A_0 = k$. Hence $R \otimes R$ is a
domain by \cite[Theorem 6.6]{Zh1}. Similarly, $R\otimes R\otimes R$ 
is a connected Hopf algebra domain. When $R$ is connected, so is 
the associated graded Hopf algebra $\gr H$ with respect to its 
coradical filtration. Therefore $(\gr R)\otimes (\gr R)\otimes (\gr R)$
is a domain by the above argument.

\begin{prop}
Let $k$ be algebraically closed of characteristic 0. Let $R$ be
a connected Hopf $k$-algebra and let $T = R[x; \sigma, \delta]$ be
a Hopf algebra containing $R$ as a Hopf subalgebra. Then
$$\Delta(x)=1\otimes x+x\otimes 1+w$$
for some $w\in R\otimes R$. As a consequence, $T$ is a HOE of $R$
and is a connected Hopf algebra.
\end{prop}

\begin{proof} By comments before Proposition, 
$R\otimes R$ is a domain, and by Lemma \ref{coHOE}.1,
$$\Delta(x)=s(1\otimes x)+t(x\otimes 1)+v(x\otimes x)+w$$
where $s,t,v,w\in R\otimes R$. Then
$$\begin{aligned}
(\Delta\otimes 1) &\Delta(x)=(\Delta\otimes 1)(s) (1\otimes 1\otimes x)\\
&+(\Delta\otimes 1)(t)\{(s(1\otimes x)+t(x\otimes 1)+v(x\otimes x)+w)
\otimes 1\}\\
&+(\Delta\otimes 1)(v) \{(s(1\otimes x)+t(x\otimes 1)+v(x\otimes x)+w)
\otimes x\}+(\Delta\otimes 1)(w),
\end{aligned}$$
while
$$\begin{aligned}
(1\otimes \Delta) &\Delta(x)=(1\otimes \Delta)(s)(1\otimes
(s(1\otimes x)+t(x\otimes 1)+v(x\otimes x)+w))\\
&+(1\otimes \Delta)(t)(x\otimes 1\otimes 1)\\
&+(1\otimes \Delta)(v)\{x\otimes (s(1\otimes x)+t(x\otimes 1)+v(x\otimes x)+w)\}
+(1\otimes \Delta)(w).
\end{aligned}
$$
Comparing the coefficients of the term $x\otimes x\otimes x$ in the
coassociative equation,
$$(\Delta\otimes 1) \Delta(x)=(1\otimes \Delta) \Delta(x)$$
we obtain
\begin{equation}
\label{addeq1}
(\Delta\otimes 1)(v) (v\otimes 1)=(1\otimes \Delta)(v)(1\otimes v).
\end{equation}
Similarly, by comparing the coefficients in the other terms
(such as $1\otimes 1\otimes 1$, $x\otimes 1\otimes 1$, etc) in the
above coassociative law for $x$, we have a list of equations
\begin{align}
\label{addeq2}
(\Delta\otimes 1)(t) (w\otimes 1)+(\Delta\otimes 1)(w)&=
(1\otimes \Delta)(s)(1\otimes w)+(1\otimes \Delta)(w),\\
\label{addeq3}
(\Delta\otimes 1)(t)(t\otimes 1)&=
(1\otimes \Delta)(t)+(1\otimes \Delta)(v)(1\otimes w),\\
\label{addeq4}
(\Delta\otimes 1)(t) (s\otimes 1)&=
(1\otimes \Delta)(s) (1\otimes t),\\
\label{addeq5}
(\Delta\otimes 1)(s)+(\Delta\otimes 1)(v) (w\otimes 1)&=
(1\otimes \Delta)(s) (1\otimes s),\\
\label{addeq6}
(\Delta\otimes 1)(t) (v\otimes 1)&=(1\otimes \Delta)(v) (1\otimes t),\\
\label{addeq7}
(\Delta\otimes 1)(v) (t\otimes 1)&=
(1\otimes \Delta)(v) (1\otimes s),\\
\label{addeq8}
(\Delta\otimes 1)(v) (s\otimes 1)&=(1\otimes \Delta)(s) (1\otimes v).
\end{align}

We claim that the equation \eqref{addeq1} implies that
$v$ is $c 1\otimes 1$ for some scalar $c\in k$. Let $R_n$ be the
$n$th term in the coradical filtration on $R$. We define an
${\mathbb N}^3$-filtration on $R\otimes R\otimes R$ by
$$F_{l,m,n}=R_{l}\otimes R_{m}\otimes R_{n}$$
for all $l,m,n\in {\mathbb N}$. Note that ${\mathbb N}^3$ is an ordered
semigroup with an total ordering defined as follows:

\begin{enumerate}
\item[]
$(l_1,m_1,n_1)<(l_2,m_2,n_2)$ if and only if either
$l_1+m_1+n_1<l_2+m_2+n_2$, or $l_1+m_1+n_1=l_2+m_2+n_2$ and
$l_1<l_2$, or $l_1+m_1+n_1=l_2+m_2+n_2$ and
$l_1=l_2$ and $m_1<m_2$.
\end{enumerate}

\noindent
Define $F_{\leq (l,m,n)}=\sum_{(l',m',n')\leq (l,m,n)} F_{l',m',n'}$. Then
$\{F_{\leq (l,m,n)}\mid (l,m,n)\in {\mathbb N}^3\}$ is a filtration
such that the associated graded ring is $(\gr R)\otimes (\gr R)
\otimes (\gr R)$, which is a domain (see comments before 
Proposition). Similarly,
we can define a filtration on $R\otimes R$. Using these filtrations,
the degree of elements in $R^{\otimes d}$, for $d=1,2,3$, is well-defined.
If $f\in R$ has degree $a$, then $\deg \Delta(f)= (a,0)$, by \cite[Theorem 5.2.2(2)]{Mon} and the counit axiom.
Similarly, if $\deg f=(a,b)$ for some $f\in R\otimes R$, then $\deg
(\Delta \otimes 1)(f)=(a,0,b)$. Let $\deg v=(m,n)$, and apply
$\deg $ to the equation \eqref{addeq1}. Using the fact that $(\gr R)\otimes (\gr R)
\otimes (\gr R)$ is a domain, we obtain
$$(m,0,n)+(m,n,0)=(m,n,0)+(0,m,n)$$
which implies that $m=0$. So, since $R_0 = k$, $v=1\otimes f$ for some $f\in R$. By
symmetry, $v=g\otimes 1$ for some $g\in R$. Hence $v=c 1\otimes 1$
for some $c\in k$ as desired.

Suppose that $v\neq 0$; that is, that $c \in k \setminus\{0\}$. Then \eqref{addeq8} is
equivalent to $s\otimes 1=(1\otimes \Delta)(s)$, which
implies that $s=a\otimes 1$ for some $a\in R$. Similarly,
\eqref{addeq6} implies that $ t=1\otimes b$ for some
$b\in R$. Therefore \eqref{defnHOE3} holds and $T$ is a HOE of $R$.
Since $R$ is a domain, by Lemma  \ref{anti}
and Theorem {\rm{\ref{prowinding}(i)}}, $v=0$, a contradiction.
Therefore $v=0$.

With $v=0$, \eqref{addeq3} becomes
$$(\Delta\otimes 1)(t) (t\otimes 1)=(1\otimes \Delta)(t).$$
Let $\deg t=(m,n)$. Considering degrees in the above
equation, we have
$$(m,0,n)+(m,n,0)=(m,n,0)$$
which implies that $m=n=0$. So $t=d 1\otimes 1$ for some $d=d^2\in k$.
Since $d\neq 0$ by the counit axiom, $d=1$ and $t=1\otimes 1$.
By symmetry, $s=1\otimes 1$. The result follows.

By Proposition \ref{coradical}, $T_0=R_0=k$, so $T$ is connected.
\end{proof}

The following corollary is immediate.

\begin{cor}
Let $k$ be algebraically closed of characteristic 0, and $R$
a connected Hopf $k$-algebra and let $T = R[x; \sigma, \delta]$ be
a Hopf algebra containing $R$ as a Hopf subalgebra. Then the
conclusions of Theorem {\rm{\ref{prowinding}(i)}},
Proposition {\rm{\ref{coradical}}} and
Theorem {\rm{\ref{GKdim}}}, are valid for $T$.
\end{cor}

\section{Iterated Hopf Ore extensions} \label{iterOre}
\subsection{Definition and first examples}\label{iterdefn}

By Proposition \ref{skewconn}, \eqref{defnHOE3} is automatic for $T:=R[x;\sigma,\delta]$
when $R$ is connected, and in this case $T$ is connected. So we can make the following
definition without mentioning requirement \eqref{defnHOE3}. 

\begin{defn} An \emph{iterated Hopf Ore extension of $k$} (IHOE) is a
Hopf $k$-algebra \begin{equation} \label{IHOEchain} H =
k[X_1][X_2; \sigma_2 , \delta_2] \dots [X_n; \sigma_n , \delta_n],
\end{equation} where
\begin{itemize}
\item
$H$ is a Hopf $k$-algebra;
\item
$H_{(i)} := k\langle X_1,
\ldots , X_i \rangle$ is a Hopf subalgebra of $H$ for $i = 1,
\ldots , n;$
\item
$\sigma_i$ is an algebra automorphism of
$H_{(i-1)}$, and $\delta_i$ is a $\sigma_i$-derivation of
$H_{(i-1)}$, for $i = 2, \ldots , n$.
\end{itemize}
\end{defn}

\begin{exs}\label{IHex}
\begin{enumerate}
\item Let $\mathcal{O}(G)$ be the coordinate ring
  of the $k$-affine algebraic group $G$. Recall that $G$ is
\emph{unipotent} if, for all $g \in G$, $g-1$ acts nilpotently on
all finite dimensional (rational) representations. The following
are equivalent.
  \begin{enumerate}
  \item[(a)] $\mathcal{O}(G)$ is an IHOE;
  \item[(b)] $\mathcal{O}(G)$ is a polynomial $k$-algebra;
  \item[(c)] $\mathcal{O}(G)$ is a connected Hopf algebra;
\item[(d)] $G$ is unipotent.
\end{enumerate}
In this case $n = \mathrm{dim}G$.

Here, $(a) \Rightarrow (b)$ is clear,and  $(b) \Rightarrow (d)$ is
a theorem of Lazard, \cite{Laz}, see also \cite{KP}. $(d) \Rightarrow (a)$
follows since a unipotent group in characteristic 0 is a subgroup
of the group of strictly upper triangular $n \times n$ matrices,
for some $n$, \cite[Corollary 17.5]{Hu}, and so has a finite normal
(even central) series with successive factors isomorphic to
$(k,+)$. That these equivalent conditions imply (c) is a well-known
fact in the theory of algebraic groups, namely, that the trivial
module is the only simple rational $G$-module when $G$ is unipotent.
But it is in any case a consequence of Proposition \ref{coradHOE}.
Finally, if (c) holds, then by duality the trivial module is the
unique rational simple $G$-representation, from which it follows
from the definition of a unipotent group that $G$ is unipotent.
\item
The enveloping algebra $U(\mathfrak{g})$ of the finite
dimensional solvable Lie algebra $\mathfrak{g}$ is an IHOE, with
$n = \mathrm{dim}_k(\mathfrak{g})$ and $\sigma_i = \mathrm{id}$
for all $i$. This follows from the fact that $\mathfrak{g}$ has a
chain of ideals $\mathfrak{g}_i,\, 0 \leq i \leq n,$, with
$\mathfrak{g}_i \subset \mathfrak{g}_{i+1}$ and
$\mathrm{dim}_k(\mathfrak{g}_i)= i,$ for all $i$, \cite[1.3.14]{Di}.
\item
Let $\mathfrak{g} = \mathfrak{sl}(2,k) = ke \oplus kh \oplus
kf$, with the standard relations. Then $H := U(\mathfrak{g})$ is
an IHOE,
$$ H = k[h][e;\sigma_2][f;\sigma_3, \delta_3], $$
with $\sigma_2 (h) = h-2,$ $\sigma_3 (e) = e$ and $\sigma_3 (h) =
h+2,$ and $\delta_3$ mapping $h$ to 0 and $e$ to $-h$.
\item
If $\mathfrak{g}$ is a semisimple Lie $k$-algebra with a
simple factor not isomorphic to $\mathfrak{sl}(2,k)$, then
$U(\mathfrak{g})$ is \emph{not} an IHOE. To see this, note that
the Hopf subalgebras of $U(\mathfrak{g})$ are the enveloping
algebras of the Lie subalgebras of $\mathfrak{g}$, and there are
insufficient of these to form a full flag in $\mathfrak{g}$.
\end{enumerate}
\end{exs}

\subsection{First properties of IHOEs}
\label{firstprop}

Clearly, all the results of \S \ref{HopfOre} can be applied to
IHOEs, passing inductively up the chain (\ref{IHOEchain}). Note
that the issue of $\S$\ref{tensor}, whether $R \otimes R$ is a
domain, is easily resolved for an IHOE $R$, since $R \otimes R$ is
again an IHOE. Let us write $P(H)$ to denote the space of
primitive elements of the Hopf algebra $H$. Definitions and
background for the terminology introduced in (v) and (vii) of
the following result can be found in many places, see for
example \cite{BZ}.

\begin{thrm}
Let $H$ be an IHOE with defining chain
\eqref{IHOEchain}.
\begin{enumerate}
\item
$H$ is a connected Hopf algebra with, for each $i = 1, \ldots , n,$
\begin{equation} \label{hitch}
\Delta (x_i) = x_i \otimes 1 + 1 \otimes x_i +
w^{i-1},
\end{equation}
where $w^{i-1} \in H_{(i-1)} \otimes
H_{(i-1)}$, for $i = 1, \ldots , n$, with $w^0 = w^1 = 0$. After
changes of the defining variables $\{x_i\}$ but not of the chain
\eqref{IHOEchain}, the data $\{x_i,\sigma_j, \delta_j, w^{i-1} :
2 \leq j \leq n,\, 1 \leq i \leq n \}$
satisfies {\rm{(}}appropriate formulations of{\rm{)}} the conditions
listed in Theorem {\rm{\ref{prowinding}}}, with $a = 1.$
\item
$H$ is a noetherian domain of GK-dimension $n$.
\item
After further variable changes {\rm{(}}not affecting the validity of
{\rm{(i))}}, the Lie algebra $P(H)$ is contained in the space
$\sum_{i=1}^n k x_i,$ with equality if and only if $H$ is
isomorphic as a Hopf algebra to the enveloping algebra $U(P(H))$.
\item The associated graded Hopf algebra of $H$ with respect to the
coradical filtration is a commutative polynomial algebra in $n$
variables. \item $H$ is Auslander-regular and AS-regular of
dimension $n$, and is GK-Cohen-Macaulay.
\item
$H$ has Krull dimension at most $n$. \item $H$ is skew Calabi-Yau with Nakayama
automorphism $\nu$, where, for $i = 1, \ldots , n$,
$$ \nu (x_i) = x_i + a_i,$$
with $a_i \in H_{(i-1)}.$ If $x_i \in P(H)$ then $a_i \in k.$
\item
The antipode $S$ of $H$ either has infinite order, or $S^2 = \mathrm{Id}$.
\item
$S^4 = \tau^{\ell}_{\chi}\circ \tau^r_{-\chi}$, the composition of a
left winding automorphism of $H$ with the right winding automorphism
of the inverse character. The character $\chi$ belongs to the centre
of the character group $X(H)$. In particular, $S^4$ is a unipotent
automorphism of a generating finite dimensional subcoalgebra of $H$.
\end{enumerate}
\end{thrm}

\begin{proof}
(i) This follows by induction on $n$, using Proposition \ref{skewconn}
to handle the induction step.

(ii) That $H$ is a noetherian domain is immediate from basic properties
of Ore extensions, \cite[Theorem 1.2.9]{MR}. In view of (i),
Theorem \ref{GKdim} applies at every step of the defining chain
(\ref{IHOEchain}) for $H$, and shows that $\mathrm{GKdim}H = n.$

(iii) We argue by induction on the GK-dimension $n$ of $H$, the
case $n =0$ being trivial. Suppose that $n \geq 1$, so that
$$ H = T[x; \sigma, \delta], $$
where $T$ is an IHOE of GK-dimension $n-1$ that is constructed using the
variables $x_1, \ldots ,x_{n-1}$, with
$$ T_1 = H_1 \cap T \subseteq \sum_{i=1}^{n-1}k x_i. $$
If $H_1 \subseteq T$ then there is nothing to prove. So, suppose that
$H_1$ is not contained in $T$ and choose
$w \in \sum_{i=0}^m t_i x^i \in H_1 \setminus T$, with
$m := \mathrm{deg} w$ minimal. Let $t_m \in T_{\ell}$, and suppose
that $\ell \geq 1$. By \cite[Lemma 5.3.2]{Mon},
$$ \Delta (t_m) = t_m \otimes 1 + 1 \otimes t_m + \gamma, $$
for an element $\gamma \in T_{\ell - 1} \otimes T_{\ell -1}.$

Now, by (\ref{hitch}), for some $\eta \in T \otimes T$,
\begin{align*} \Delta (w) &= \sum_{i=0}^m\Delta(t_i)\Delta (x)^i \\
&= \sum_{i=0}^m\Delta(t_i)(x \otimes 1 + 1 \otimes x + \eta)^i \\
&= (t_m \otimes 1 + 1 \otimes t_m + \gamma)
(x \otimes 1 + 1 \otimes x + \eta)^m + \beta \\
&= t_m x^m \otimes 1 + 1 \otimes t_m x^m + x^m \otimes t_m
+ t_m \otimes x^m + \mu,
\end{align*}
where $\beta , \mu \in H \otimes H$ do not involve any right or
left tensorands of degree in $x$ greater than or equal to $m$.

But $w \in P(H)$, so
\begin{equation} \label{hip}
\Delta (w) = t_m x^m \otimes 1 + 1 \otimes t_m x^m +
\sum_{i=0}^{m-1}t_i x^i \otimes 1 + \sum_{i=0}^{m-1}1 \otimes t_i x^i.
\end{equation}
The two expressions for $\Delta (w)$ cannot be reconciled, so we
conclude that $t_m \in k$. Therefore, without loss of generality,
$ t_m = 1.$ Suppose now that $m > 1.$ Then
\begin{align*}
\Delta (w) &= \Delta (x)^m +  \sum_{i=0}^{m-1}\Delta(t_i)\Delta (x)^i \\
&= x^m \otimes 1 + 1 \otimes x^m + mx^{m-1} \otimes x + \zeta,
\end{align*}
where $\zeta \in H \otimes H$ involves no terms in $x^{m-1} \otimes x$.
Comparing this with (\ref{hip}) we are forced to conclude that $m =1 $.
That is,
$$ w = x + t,$$
for some element $t$ of $T$.

As is well-known, we can now change the $n$th variable in the IHOE
$H$ from $x$ to $w$, replacing $\delta$ by $\hat{\delta} := \delta
+ [t_{\sigma}, - ]$, where $[t_{\sigma}, r] := tr - \sigma (r)t$ for
$r \in R.$ Then $\hat{\delta}$ is a $\sigma$-derivation of $T$ and
$H = T[w; \sigma, \hat{\delta}]$, so that the induction step is proved.
This proves the first claim in (iii).

To prove the second statement in (iii), suppose first that all the
$x_i$ are primitive, then $H$, being generated by these elements, is
cocommutative. The only units of an iterated skew polynomial algebra
being scalars, $G(H) = \{1\}$. Since we are in characteristic 0, a
special case of the Cartier-Gabriel-Kostant theorem,
\cite[Theorem 5.6.5]{Mon}, ensures that $H$ is the enveloping algebra
of the Lie algebra $P(H) = \sum{i=1}^n kx_i$. Conversely, suppose
that $H$ is an IHOE which is isomorphic as Hopf algebra to the
enveloping algebra of the Lie algebra $\mathfrak{g}$. We claim that
$H$ can be realised as an IHOE with a basis of $\mathfrak{g}$ as the
set of skew polynomial generators of $H$. Argue by induction on
$n := \mathrm{GKdim} H = \mathrm{dim}_k \mathfrak{g}$, the second
equality holding by \cite[Example 6.9]{KL}. By hypothesis,
$H = T[x;\sigma,\delta]$ for an IHOE $T$, with $\mathrm{GKdim} T = n-1$,
by Corollary (\ref{skewconn}). Since Hopf subalgebras of enveloping
algebras are enveloping algebras (by, for example,
\cite[Theorem 5.6.5]{Mon} again), $T \cong U(\mathfrak{h})$ for an
$n-1$-dimensional Lie subalgebra of $\mathfrak{g}$. By induction,
$T$ can be realised as an IHOE with a basis of $\mathfrak{h}$ as skew
polynomial generating set. By \cite[Proposition 5.5.3(2)]{Mon},
$$ P(T) = \mathfrak{h} \subsetneq \mathfrak{g} = P(H).$$
The argument from the first part of the proof of (iii) can now be
used to show that there exists $t \in T$ with $x-t \in P(H)$; and
as before, $H = T[x-t;\sigma, \hat{\delta}]$, as required.

(iv) Immediate from (i),(ii) and \cite[Theorem 6.9]{Zh1}.

(v),(vi) This all follows by standard techniques from
(iv) - see \cite[Corollary 6.10]{Zh1}.

(vii) By \cite[Proposition 4.5]{BZ} and (v), $H$ is skew Calabi-Yau
with Nakayama automorphism $\nu = S^2\tau$, where $\tau$ is a certain
left winding automorphism of $H$. For $i = 1, \ldots , n$, that
$\tau(x_i)$ and $S^2 (x_i)$ have the form ascribed to $\nu$ in (vii)
follows from (\ref{hitch}) and from Theorem \ref{permit}(i)(c)
respectively. Since $H_{(i-1)}$ is a Hopf subalgebra of $H$, the
composition of these maps also has the desired form. The second
claim is clear from the formulae for $S^2$ and $\tau$ when $w^{i-1} = 0.$

(viii) Suppose that $|S| < \infty$. We show that $S^2 = \mathrm{Id}$,
arguing by induction on $n = \mathrm{GKdim}H.$ The case $n = 0$ being
trivial, we have $H = R[x;\sigma,\delta],$ where $R$ is an IHOE of
GK-dimension $n-1$, whose antipode $S_{|R}$ has finite order, and
hence satisfies $S_{|R}^2 = \mathrm{Id}$ by the induction hypothesis.
By Theorem \ref{permit}(i)(c), there exists $r \in R$ such that
$S(x) = -x + r$. Setting $a := S(r) - r \in R$, we find
\begin{equation}\label{near} S^2(x) = x + a.\end{equation}
Continuing, we easily calculate that, for every $m \geq 1$,
\begin{equation}\label{done} S^{2m} (x) = x + ma. \end{equation}
But $S$ is supposed to have finite order, and $k$ has characteristic 0,
so (\ref{done}) is only possible if $a = 0.$ By (\ref{near})
$$S^2(x) = x,$$
and the induction is proved.

(ix) Given the AS-regularity of $H$ from (v), the first claim follows
from \cite[Corollary 4.6]{BZ}, in view of the fact that $H$ has no
non-trivial inner automorphisms, by (ii). (The character $\chi$ is
given by the right action of $H$ on the left homological integral
$\mathrm{Ext}^n_H(k,H)$.) Let $I(H) = \langle [H,H] \rangle$. As
discussed in detail in $\S$\ref{classic} below, the maximal commutative
factor $H/I(H)$ of $H$ is a Hopf algebra. Since $H/I(H)$ is affine and
commutative, it is the coordinate ring of an affine algebraic $k$-group
$X(H)$, the character group of $H$. Since $I(H)$ is a Hopf ideal of $H$
it is fixed by $S^4$, and $S^4$ then induces on $H/I(H)$ the fourth
power of the antipode of $H/I(H)$. But the square of the antipode of a
commutative Hopf algebra is the identity, \cite[Corollary 1.5.12]{Mon},
$S^4$ induces the identity on $H/I(H)$. The centrality of $\chi$ is
immediate.

Being a factor of the connected algebra $H$, $H/I(H)$ is itself
connected by \cite[Corollary 5.3.5]{Mon}. Thus, by Examples
\ref{iterdefn}(i), $X(H)$ is unipotent. Let $V$ be a finite
dimensional generating coalgebra of $H$. Since the maps $\tau^{\ell}$
and $\tau^r$ from $X(H)$ to $GL(V)$ yield commuting rational actions
of $X(H)$ on $V$, $\langle \tau^{\ell}(X(H)),\tau^r (X(H)) \rangle$
is a unipotent subgroup of $GL(V)$, as claimed.
\end{proof}

\begin{remarks}
\begin{enumerate}
\item
Both alternatives in (viii) can occur: see \ref{smallexp}(iii) for
an example with $|S| = \infty.$
\item
From the second part of (ix) of the theorem and the fact that
unipotent groups in characteristic 0 are torsion free
(Examples \ref{iterdefn}(i)), we can directly deduce a weaker
version of (viii): $S^4$ is either the identity or has infinite order.
\item
The inequality in (vi) can be strict: the Krull dimension of
$U(\mathfrak{sl}(2,k)$ is 2, by \cite{Le}.
\item
Further information about the character group $U$ appearing in
(ix) is obtained in $\S$\ref{classicsec}.
\end{enumerate}
\end{remarks}

\subsection{IHOEs of small dimension}
\label{small}
If $H$ is a connected Hopf $k$-algebra of finite GK-dimension, then
\begin{equation}\label{GKineq}
\mathrm{min}\{2, \mathrm{GKdim}H\} \leq \mathrm{dim}_k P(H) \leq
\mathrm{GKdim}H,
\end{equation}
by \cite[Lemma 5.11]{Zh2}. Moreover, if $K$ is a proper Hopf
subalgebra of $H$ then $\mathrm{GKdim}K < \mathrm{GKdim} H$, by
\cite[Lemma 7.2]{Zh1}. Parts (i), (ii) and (iii) of the following
theorem follow easily from these facts. The connected Hopf
$k$-algebras of GK-dimension 3
[resp. 4] have been classified by Zhuang \cite{Zh1}
[resp. by Wang, Zhang and Zhuang
\cite{WZZ}]. In view of Theorem \ref{firstprop}(i)(ii), these
classifications yield all IHOEs of dimension at most four.

We summarise the picture in the following theorem, summarising
results from \cite{Zh1} and \cite{WZZ}. The terminology and
notation appearing in (iv) is explained in $\S$\ref{smallexp},
after the theorem. For the definitions of the terms used in
(v)(b) and (v)(c), see \cite{WZZ}.

\begin{thrm} Let $H$ be an IHOE.
\begin{enumerate}
\item
If $\mathrm{dim}_k H$ is finite, then $H = k$.
\item
If $\mathrm{GKdim}H= 1$ then $H = k[x].$
\item
If $\mathrm{GKdim}H= 2$ then $H$ is the enveloping algebra of one of
the two Lie algebras of dimension 2.
\item
Suppose that $\mathrm{GKdim} H = 3.$ Then $H$ is isomorphic as a
Hopf algebra to one (and only one) of the following:
\begin{enumerate}
\item[(a)]
the enveloping algebra of a three-dimensional Lie algebra;
\item[(b)]
the algebras $A(0,0,0), A(0,0,1), A(1,1,1),
A(1,\lambda,0)$, $\lambda \in k$;
\item[(c)]
the algebras $B(\lambda)$, $\lambda \in k.$
\end{enumerate}
\item
Suppose that $\mathrm{GKdim} H = 4.$ Then $H$ is isomorphic
as a Hopf algebra to one (and only one) of the following:
\begin{enumerate}
\item[(a)]
the enveloping algebra of a four-dimensional Lie algebra;
\item[(b)]
the enveloping algebra $U(\mathfrak{g})$ of an anti-commutative
coassociative Lie algebra $\mathfrak{g}$ of dimension 4;
\item [(c)]
a primitively-thin Hopf algebra of GK-dimension 4.
\end{enumerate}
\end{enumerate}
In particular, every connected Hopf $k$-algebra of GK-dimension at
most 4 is an IHOE.
\end{thrm}

\subsection{Definitions and remarks for Theorem \ref{small}}
\label{smallexp}

(i) To explain (iv) we need to explain, following \cite[$\S$ 7]{Zh1},
but adapting the presentations of \cite[Lemma 3.2]{WZZ} the families
of algebras $A(\lambda_1,\lambda_2, \alpha)$ and $B(\lambda)$.

Let $A$ be the factor of the free algebra $k\langle
X,Y,Z \rangle$ by the ideal generated by
$$ [X,Y]$$
$$[Z,X] - \lambda_1X + \alpha Y,$$
$$ [Z,Y] - \lambda_2 Y, $$
where $\alpha = 0$ if $\lambda_1 \neq \lambda_2$, and $\alpha = 0$
or $1$ if $\lambda_1 = \lambda_2$. Abusing notation by writing
$X,Y,Z$ for the images of the corresponding elements in $A$, $A$
becomes a Hopf algebra with augmentation ideal $\langle X,Y,Z
\rangle$, with $X$ and $Y$ primitive, with $$ \Delta (Z) = Z
\otimes 1 + 1 \otimes Z + X \otimes Y - Y \otimes X$$ and with
$$ S(X) = -X, \; \; S(Y) = -Y, \; \; S(Z) = -Z. $$

Let $B$ be the factor of $k\langle X,Y,Z \rangle$ by the ideal generated by
$$ [X,Y] - Y,$$
$$ [Z,X] +Z - \lambda Y, $$
$$  [Z,Y],$$
where $\lambda \in k$. Then $B$ has a Hopf algebra structure with
the same coproduct for $kX + kY + kZ$ as
for $A$. Thus, by Theorem \ref{permit}(i)(c), the only possible values
for the antipode are $S(X) = -X$, $S(Y) = -Y$ and $S(Z) = -Z + Y$.

(ii) Zhuang shows in the proof of \cite[Prop.7.1]{Zh1} that
$A(\lambda_1, \lambda_2,\alpha)$ and $B(\lambda)$ are IHOEs, with the
following structures.

For $A$, there is the chain of Hopf subalgebras
\begin{equation}\label{AHOE}
A_0 = k \subseteq A_1 = k[Y] \subseteq A_2 = k[Y,X]
\subseteq A_2[Z; \delta] = A,
\end{equation}
where $\delta$ is the obvious (given the relations) derivation of $A_2$.

Similarly, for $B$ we can take the chain of Hopf subalgebras
\begin{equation}\label{BHOE}
B_0 = k \subseteq B_1 = k[Y] \subseteq B_2 = k\langle Y,X\rangle
\subseteq B_2[Z; \sigma, \delta] = B,
\end{equation}
where $B_2$ is the enveloping algebra of the 2-dimensional nonabelian
Lie algebra, and $\sigma$ and $\delta$ are clear from the relations.

(iii) Note that the antipode of $B(\lambda)$ has infinite order,
with $S^{2m} (Z) = Z + (2m)Y$ for all $m \geq 1$. This is in contrast
to the properties of commutative or cocommutative Hopf algebras, for
which $S^2 = \mathrm{id}$. Thus both alternatives listed in Theorem
\ref{firstprop}(viii) actually occur.

(iv) By \cite[Theorem 7.6]{Zh1}, the algebras in (iv) above are
precisely the connected Hopf $k$-algebras of GK-dimension 3. That
they are all IHOEs follows from (\ref{AHOE}), (\ref{BHOE}) and the
discussion in $\S$\ref{iterdefn}, Examples (ii) and (iii).

(v) One easily checks by a routine examination of the presentations
in \cite{WZZ} that every connected Hopf $k$-algebra of GK-dimension
4 is an IHOE.

\section{The maximal classical subgroup of a Hopf Ore extension}
\label{classicsec}
\subsection{Maximal classical subgroups.}
\label{classic}
If $T$ is a $k$-algebra, we call an algebra homomorphism
$\chi : T \longrightarrow k$ a \emph{character} of $T$, and its
kernel $\mathrm{Ker}\chi$ a \emph{character ideal}. Set $\Xi(T)$
to denote the set of character ideals of $T$. Note that if $T$ is
affine and $k$ is algebraically closed then $\Xi (T)$ has a natural
structure as an algebraic set. We write $\langle [T,T]\rangle$ to
denote the ideal of $T$ generated by the commutators
$\{ab - ba : a,b \in T\}$, and
$$ I(T) \quad := \quad \cap \{ M : M \in \Xi (T) \}, $$
so that $\langle [T,T]\rangle \subseteq I(T)$.

It is well known and easy to prove that, when $H$ is a Hopf $k$-algebra,
$\langle [H,H]\rangle$ and $I(H)$ are Hopf ideals of $H$. For us,
$k$ is algebraically closed of characteristic 0; let's assume also
that $H$ is affine. Then the nullstellensatz ensures that
$I(H)/\langle [H,H]\rangle$ is a nilpotent ideal of the commutative
affine $k$-algebra $\overline{H} := H/\langle [H,H]\rangle$. But, in
characteristic 0, such Hopf algebras are semiprime by \cite{Wa}, so
$\langle [H,H]\rangle = I(H).$ Thus $\overline{H} \cong \mathcal{O}(G)$
for some affine algebraic group $G$ over $k$. For obvious reasons we
call $G$ the \emph{maximal classical subgroup} of $H$.

The above discussion all applies if $H$ is assumed to be noetherian
rather than affine, since Molnar's theorem \cite{Mo} ensures that
commutative noetherian Hopf algebras are affine.

\subsection{Goodearl's theorem generalised for characters.}
\label{gengood}
In this subsection we consider general Ore extensions, rather than
Hopf algebras. In the main result (Theorem 3.1) of \cite{Go},
Goodearl describes the prime ideals of an Ore extension
$T = R[x;\sigma, \delta]$,
where $R$ is a commutative noetherian ring. To analyse the classical
subgroups of an
HOE we generalise the special case of Goodearl's theorem applying
to character ideals,
to the case where $R$ is a not necessarily commutative $F$-algebra.

\begin{thrm}
Let $F$ be algebraically closed and let $R$ be a
$F$-algebra. Let $\sigma$ and $\delta$ be respectively
an $F$-algebra automorphism and a $\sigma$-derivation of $R$, and
set $T = R[x; \sigma, \delta]$. Write
$\Psi: \Xi (T) \longrightarrow \Xi (R): M \mapsto M \cap R.$
\begin{enumerate}
\item
Let $M \in \Xi(T)$ and denote $\Psi(M)$ by $\mathfrak{m}.$ Then
{\rm{(a)}} $\delta( [R,R]) \subseteq \mathfrak{m}$, and
either {\rm{(b)}} $\mathfrak{m}$ is $(\sigma, \delta)$-invariant, or
{\rm{(c)}} $\mathfrak{m}$ is not $\sigma$-invariant.
\item
Let $\mathfrak{m} \in \Xi (R)$, and suppose that {\rm{(b)}}
{\rm{(}}and hence {\rm{(a))}} hold for $\mathfrak{m}$. Then
$\mathfrak{m}T \lhd T$ and $T/\mathfrak{m}T \cong k[x]$, so that
$$ \Psi^{-1} (\mathfrak{m}) = \{ \langle \mathfrak{m}T, x -
\lambda \rangle : \lambda \in k \} \cong \mathbb{A}^{1}_k.$$
\item
Let $\mathfrak{m} \in \Xi (R)$, and suppose that {\rm{(a)}} and {\rm{(c)}}
hold for $\mathfrak{m}$. Then there exists a unique $M \in \Xi (T)$
with $\Psi (M) = \mathfrak{m}.$
\end{enumerate}
\end{thrm}

In the next few pages, we will refer to the conditions (a), (b)
and (c) in Theorem \ref{gengood}(i). So we list these conditions
here explicitly
\begin{enumerate}
\item[(\ref{gengood}(ia))]
$\delta( [R,R]) \subseteq \mathfrak{m}$,
\item[(\ref{gengood}(ib))]
$\mathfrak{m}$ is $(\sigma, \delta)$-invariant,
\item[(\ref{gengood}(ic))]
$\mathfrak{m}$ is not $\sigma$-invariant.
\end{enumerate}

\subsection{Proof of Theorem \ref{gengood}: initial steps}
\label{goodproof1}
Throughout the proof $F, R,x, \sigma, \delta$ and $T$ will be as in
Theorem \ref{gengood}. We begin by observing that if $A$ is an ideal
of $T$ and $\mathfrak{a} := A \cap R,$ then
$\delta (\mathfrak{a} \cap \sigma^{-1} (\mathfrak{a}))
\subseteq \mathfrak{a}.$
This follows by applying the relation (\ref{relations}) with
$r \in \mathfrak{a} \cap \sigma^{-1}(\mathfrak{a}).$ In the
following result we address the converse to this statement,
for character ideals. In essence we construct a sort of Verma
module for $T$ by induction from a 1-dimensional $R$-module,
and show that under suitable hypotheses this module has a
1-dimensional factor.

\begin{prop}  Let $\mathfrak{m}$ be an ideal of $R$ with
$R/\mathfrak{m} \cong F.$ Suppose that \begin{equation} \label{move}
\sigma(\mathfrak{m}) \neq \mathfrak{m},\end{equation} and that
\begin{equation}
\label{deltafix}
\delta (\mathfrak{m} \cap \sigma^{-1}(\mathfrak{m})) \subseteq \mathfrak{m}.
\end{equation}
Then there is a unique ideal $M$ of $T$ such that
$M \cap R = \mathfrak{m}$. Moreover, $T/M \cong k.$
\end{prop}

\begin{proof} Assume that $\mathfrak{m}$ satisfies the stated hypotheses.

$\textbf{Step 1:}$ There exists an element $r \in \mathfrak{m}$ such
that $\sigma (r) \equiv 1 (\mathrm{mod}\,\mathfrak{m});$ moreover,
$r$ is uniquely determined and non-zero, modulo
$\mathfrak{m} \cap \sigma^{-1}(\mathfrak{m})$.

For, thanks to (\ref{move}) there are elements $a,r \in \mathfrak{m}$
such that $1 = a + \sigma(r)$, and clearly $r \notin
\sigma^{-1}(\mathfrak{m}).$ Since $\mathfrak{m}/\mathfrak{m} \cap
\sigma^{-1}(\mathfrak{m}) \cong k$, $r$ is uniquely determined
modulo $\mathfrak{m} \cap \sigma^{-1}(\mathfrak{m})$.

$\textbf{Step 2:}$ Let $r$ be as in Step 1, and define
$\lambda := \delta (r)\mathrm{mod}\,\mathfrak{m}.$ Suppose $M$ is
an ideal of $T$ with $M \cap R = \mathfrak{m}.$ Then $x + \lambda \in M.$

We have
\begin{equation}
\label{scream}
M \ni xr = \sigma (r)x + \delta (r).
\end{equation}
  There are elements $t,s \in \mathfrak{m}$ such that
$\sigma (r) = 1 + t$ and $\delta (r) = \lambda+ s$. Substituting in
(\ref{scream}) gives
$$ x + tx + \lambda + s \in M,$$
so that $x + \lambda \in M$ and Step 2 is proved.

$\textbf{Step 3:}$ Define $I$ to be the right $F[x]-$submodule of $T$,
$$ I = (x + \lambda)k[x] + \mathfrak{m}T.$$
Then $I$ is an ideal of $T$, with $T/I \cong k$ and $I \cap R = \mathfrak{m}.$

Since it is clear that $T/I \cong k$ as vector spaces and that
$\mathfrak{m} \subseteq I$, we only need to prove that $I$ is an ideal.
Also, if $I$ is a left ideal of $T$, then $IT =I(k + I) \subseteq I,$
so that $I$ is an ideal. To prove that $I$ is a left ideal, it's
sufficient to show that $\mathfrak{m}I \subseteq I$ and that
$xI \subseteq I$, since $T$ is generated by $x$ and $\mathfrak{m}$.
The first of these claims is clear, since $\mathfrak{m}T \subseteq I.$
For the second claim, since $T = \mathfrak{m}k[x] + k[x],$ it's enough
to prove that
  \begin{equation}
\label{crux}
x\mathfrak{m}k[x] \subseteq I.
\end{equation}
  Thus, let $w \in \mathfrak{m}$ and $f \in k[x]$, and consider $xwf.$
Suppose first that $w \in \mathfrak{m} \cap \sigma^{-1}(\mathfrak{m}).$
Then, by (\ref{deltafix}),
$$ xwf = \sigma (w)xf +\delta (w)f \in \mathfrak{m}T \subseteq I.$$
On the other hand, $\mathfrak{m} = kr + (\mathfrak{m} \cap \sigma^{-1}
(\mathfrak{m})$ by Step 1, so it remains only to show that
\begin{equation}\label{final?}
xrf \in I.
\end{equation}
But
\begin{align*} xrf &= \sigma (r)xf + \delta (r)f \\
&= (1 +t)xf + (\lambda + s)xf \\
&= f(x + \lambda) + txf + sf \\
&\in (x + \lambda)k[x] + \mathfrak{m}T = I. \end{align*}
\end{proof}

\subsection{Proof of Theorem \ref{gengood}: conclusion.}
\label{goodproof2}
\begin{proof} (i) Let $M \in \Xi (T)$, with $\Psi (M) := \mathfrak{m}.$
If $\sigma (\mathfrak{m}) \neq \mathfrak{m}$, then
$$ \delta( [R,R]) \subseteq \delta (\mathfrak{m} \cap
\sigma^{-1}(\mathfrak{m})) \subseteq \mathfrak{m},$$
by the observation at the start of $\S$\ref{goodproof1}. On the other
hand, if $\sigma(\mathfrak{m}) = \mathfrak{m}$ then the same
calculation using (\ref{relations}) shows that
$\delta(\mathfrak{m}) \subseteq \mathfrak{m}.$

Therefore (\ref{gengood}(ia)) holds in both cases, and either
(\ref{gengood}(ib)) or (\ref{gengood}(ic)) holds, as claimed.

(ii) Let $\mathfrak{m}$ be a $(\sigma, \delta)-$invariant character
ideal of $R$. So $\mathfrak{m}T$ is an ideal of $T$, and
$$T/\mathfrak{m}T \cong k[x].$$
Therefore (ii) follows.

(iii) Let $\mathfrak{m} \in \Xi (R)$ and suppose that $\mathfrak{m}$
satisfies (\ref{gengood}(ia)) and (\ref{gengood}(ic)). We claim
that (\ref{deltafix}) holds for $\mathfrak{m}$. First, note that
if $\mathfrak{a}$ and $\mathfrak{b}$ are distinct character ideals
of a $F$-algebra $R$, then
\begin{equation}
\label{chinese}
\mathfrak{a}\cap \mathfrak{b}
= \langle [R,R] \rangle + \mathfrak{a}\mathfrak{b}
= \langle [R,R] \rangle + \mathfrak{b}\mathfrak{a}.
\end{equation}
This follows from the inclusion
$$ \mathfrak{a} \cap \mathfrak{b} = (\mathfrak{a} \cap \mathfrak{b})R
= (\mathfrak{a} \cap \mathfrak{b})(\mathfrak{a} + \mathfrak{b})
\subseteq \mathfrak{a}\mathfrak{b} + \mathfrak{b}\mathfrak{a},$$
so that $\mathfrak{a}\mathfrak{b} + \mathfrak{b}\mathfrak{a}
= \mathfrak{a} \cap \mathfrak{b}.$ Since
$\langle [R,R] \rangle \subseteq \mathfrak{a} \cap \mathfrak{b}$,
we get $\mathfrak{a}\mathfrak{b} + \mathfrak{b}\mathfrak{a}
+ \langle [R,R] \rangle = \mathfrak{a} \cap \mathfrak{b},$
proving (\ref{chinese}).

An easy calculation shows that (\ref{gengood}(ia)) implies that
\begin{equation}\label{grate}
\delta \langle [R,R] \rangle \subseteq \mathfrak{m}.
\end{equation}
Finally, by (\ref{chinese}) and (\ref{grate}),
\begin{align*} \label{getit}
\delta (\mathfrak{m} \cap \sigma^{-1}(\mathfrak{m}))
&= \delta (R[R,R]R + \sigma^{-1}(\mathfrak{m})\mathfrak{m})\\
&\subseteq \delta (R[R,R]R) + \delta(\sigma^{-1}(\mathfrak{m}))
\mathfrak{m} + \mathfrak{m}\delta(\mathfrak{m}) \\
&\subseteq \mathfrak{m},
\end{align*}
proving (\ref{deltafix}) as claimed.

Thus the hypotheses of Proposition \ref{goodproof1} are satisfied
by $\mathfrak{m}$, and hence (iii) is proved.
\end{proof}

\begin{rems}
\begin{enumerate}
\item
Clearly, hypothesis (\ref{gengood}(ib)) on $\mathfrak{m} \in \Xi(R)$ implies
hypothesis (\ref{gengood}(ia)) for $\mathfrak{m}$; but this is not true for
hypothesis (\ref{gengood}(ic)), and we can check from
Examples \ref{IHex}(iii),
$T = U(\mathfrak{sl}(2,k)) = R[f; \sigma, \delta]$ that (\ref{gengood}(ia)) is
genuinely needed when $R$ is noncommutative.
\item
It is natural to ask the following question, perhaps starting
with the case of completely prime ideals:
\begin{question} Does Goodearl's theorem generalise in full to the
setting where the coefficient algebra $R$ is not commutative?
\end{question}
\end{enumerate}
\end{rems}

\subsection{Characters of Hopf Ore extensions.}
\label{HOEchar}
Retain the setting of \ref{gengood}, but assume now in addition that
$T$ is a Hopf algebra and $R$ is a right [resp. left] coideal
subalgebra of $T$. We can use the left [resp. right] winding
automorphisms of $T$, \cite[I.9.25]{BG}, to move between any two
character ideals of $T$ whilst preserving the inclusion $R \subseteq T$.
The following then follows easily from Theorem \ref{gengood}, since
one simply has to note that $\Xi (T)$ is permuted transitively by the
left (and by the right) winding automorphisms of $T$, these being
effectively the regular representations of the character group. We
say that $M \in \Xi(T)$ (with $\mathfrak{m} =  \Psi(M)$) has
\emph{invariant type} if (\ref{gengood}(ib)) holds for $\mathfrak{m}$;
and that $M$ has \emph{variant type} if (\ref{gengood}(ia)) and
(\ref{gengood}(ic)) hold for $\mathfrak{m}$.

\begin{thrm}
Let $T = R[x; \sigma, \delta]$ as in Theorem {\rm{\ref{gengood}}}. Assume
that $T$ is a Hopf $F$-algebra with $R$ a right or left coideal
subalgebra of $T$. Then either all of $\Xi (T)$ is of invariant type,
or all of $\Xi (T)$ is  of variant type.
\end{thrm}

In the first situation of the theorem we say that $R \subseteq T$ is
an \emph{invariant extension}; and in the second, we call
$R \subseteq T$ a \emph{variant extension}.

Both variant and invariant extensions occur. A convenient example is
the enveloping algebra of the two-dimensional nonabelian Lie algebra,
$T = k\langle x,y : [y,x] = x \rangle$. This Hopf algebra can be
presented as
$$ T \quad \cong \quad k[x][y; \delta],$$
where $\delta$ is the derivation $x \frac{\partial}{\partial x}$; and as
$$ T \quad \cong \quad k[y][x; \sigma ], $$
where $\sigma$ is the automorphism $\sigma(y) = y - 1.$ Thus the first
presentation is of invariant type, with
$$ \Xi (T) = \Psi^{-1}(\langle x \rangle)\approx \mathrm{Maxspec}(k[y]);$$
and the second is of variant type, with
$$ \Xi (T) = \Psi^{-1}(\mathrm{Maxspec}(k[y]) =
\{\langle y - \lambda, x \rangle : \lambda \in k \}.$$
Of course, when $\sigma = \mathrm{id}_R$ then only invariant type
is possible, but when $\sigma$ is non-trivial then \emph{a priori}
either option can occur.

\subsection{The maximal classical subgroup of an IHOE}
\label{maxIHOE}
The coradical of a factor Hopf algebra of a Hopf $k$-algebra $H$ is
contained in the image of the coradical $H_0$ of $H$ by
\cite[Corollary 5.3.5]{Mon}. Suppose that $H$ is an IHOE of GK-dimension
$n$. Then $H$ is connected by Theorem \ref{firstprop}(i). So, by the
first sentence, every Hopf factor of $H$ is also connected of finite
GK-dimension; in particular this is true of its maximal classical
subgroup $H/I(H)$. Continuing the discussion of $\S$ \ref{classic}
for $H$, recall part of the content of Examples \ref{iterdefn}(i):
the commutative IHOEs over $k$ are precisely the coordinate algebras
of the unipotent algebraic groups over $k$, and these are precisely
the commutative affine connected Hopf $k$-algebras. We can therefore
conclude that
\begin{equation}
\textit{ the maximal classical subgroup $U$ of an IHOE $H$ is unipotent.}
\end{equation}
Trivially, the dimension of $U$ is at most $\mathrm{GKdim}H$, with
equality if and only if $H = \mathcal{O}(U).$ Of course, every unipotent
group $U$ occurs as the maximal classical subgroup of an IHOE, namely
of the IHOE $\mathcal{O}(U)$, Examples \ref{iterdefn}(i).

Contrary to what one might conjecture, guided for example by the case
of enveloping algebras,
$$ \textit{ the maximal classical subgroup of an IHOE is not
necessarily normal.} $$
Recall that, in the context of algebraic groups, the normality of a
subgroup $N$ of the algebraic $k$-group $G$ is equivalent to the
normality of the defining ideal of $N$ in $\mathcal{O}(G)$,
\cite[page 36]{Mon}; normality of  Hopf ideals is defined at
\cite[Definition 3.4.5]{Mon}. Consider for example the algebras
$B(\lambda)$ of Theorem \ref{small}(iv)(c): we easily calculate that
$$I(B(\lambda)) = \langle Y,Z \rangle, $$
and it is routine to check that this ideal is
\emph{not} normal.

\section{Algebra structure of HOEs - first steps}
\label{algHOEs}
\subsection{Invariant HOEs.}
\label{invHOEs}
Recall the definitions of \emph{invariant} and \emph{variant} HOEs
in $\S$\ref{HOEchar}.

\begin{thrm} Let $R$ be a Hopf $k$-algebra and let $T =
R[x; \sigma, \delta]$ be an invariant HOE. Then $R$ is a normal
Hopf subalgebra of $T$ and there is a change of variables so that
$T = R[\tilde{x}; \partial]$, where $\partial$ is a derivation
of $R$.
\end{thrm}

\begin{proof} Since the extension $R \subseteq T$ is invariant,
$R^+T$ is a Hopf ideal of $T$, with $T/R^+ T \cong \mathcal{O}((k,+))$.
Clearly, $T$ is a free right and left $R$-module. Writing $\pi$ for
the canonical homomorphism from $T$ to $\mathcal{O}((k,+))$, we can
therefore apply \cite[Theorem 1]{Ta}, to conclude that
\begin{equation}
\label{hep}
T^{\mathrm{co}\pi} = ^{\mathrm{co}\pi}T = R.
\end{equation}
By \cite[Proposition 3.4.3]{Mon}, $R$ is a normal Hopf subalgebra
of $T$. Finally, by \cite[Theorem 8.3]{GZ},
\begin{equation}
\label{cat}
T \cong T^{\mathrm{co}\pi}[\tilde{x}; \partial],
\end{equation}
for a $k$-derivation $\partial$ of $T^{\mathrm{co}\pi}$.
Combining (\ref{hep}) and (\ref{cat}) completes the proof.
\end{proof}

\begin{rems}
\begin{enumerate}
\item
In the special case where $R \otimes R$ is a domain, most of the
above theorem follows very easily from Theorem \ref{permit}, since
that result forces $\sigma$ to be a left winding automorphism of $R$,
which can't fix $R^+$ unless $\sigma = \mathrm{id}$. It follows that
in this case $\delta = \partial$.
\item
Effectively what \cite[Theorem 8.3]{GZ} is proving that the extension
$R \subseteq T$ is cleft, in the language  of \cite[$\S$7.2]{Mon}.
\end{enumerate}
\end{rems}

\subsection{Variant HOEs over commutative coefficient rings.}
\label{varHOEs}

\begin{thrm} Let $R$ be an affine commutative
Hopf $k$-algebra, and an integral domain, so $R$ is the ring of
regular functions of the connected algebraic $k$-group $G$. Let $T
= R[x; \sigma, \delta]$ be a HOE of variant type.
\begin{enumerate}
\item
There is a change of variables such that $T = R[\tilde{x}; \sigma]$.
\item
The indeterminate $\tilde{x}$ is skew primitive, and $\sigma$ is a
winding automorphism of $R$ corresponding to a central element of $G$.
\item
\cite{Pa} Given a Hopf $k$-algebra $R$, and $\sigma, \tilde{x}$
satisfying the conditions in {\rm{(ii)}}, $T = R[\tilde{x}; \sigma]$ is
a Hopf algebra with $R$ as a Hopf subalgebra.
\end{enumerate}
\end{thrm}

Although the theorem requires $R$ to be commutative, we state and
prove the preparatory lemma below under more general hypotheses. Note
that the hypothesis imposed on $\sigma$, that it is a winding
automorphism of $R$, always holds if $R \otimes R$ is a domain, by
Theorem \ref{permit}. Recall the notation concerning characters
introduced in $\S$\ref{classic}.

\begin{lemma}
Let $R$ be an affine or noetherian Hopf $k$-algebra, and let
$T = R[x;\sigma, \delta]$ be a variant HOE of $R$. Assume that
$\sigma$ is a left winding automorphism of $R$.
\begin{enumerate}
\item
$\Xi (T) \approx  \Xi(R/\langle \delta([R,R]) + [R,R] \rangle). $
\item
$R + I(T) = T.$
\end{enumerate}
\end{lemma}

\begin{proof}(i) Since the HOE is variant, $\sigma (R^+) \neq R^+$.
But $\sigma$ is thus a non-trivial winding automorphism, so
\emph{every} character of $R$ is moved by it. Hence (i) follows
from Theorems \ref{gengood}(i),(iii) and \ref{HOEchar}.

(ii) Note that if $R$ is noetherian then the commutative Hopf
algebra $R/I(R)$ is affine by Molnar's theorem \cite{Mo}. Hence
$T/I(T)$ is affine under either hypothesis on $R$. We show first that
\begin{equation}\label{counitepi}
I(T) + R^+ T = T^+ .
\end{equation}
The inclusion of the left side in the right is clear. Indeed,
by Theorem \ref{gengood}, $T^+$ must be the unique character ideal
of $T$ containing $R^+$. Now $I(T) + R^+ T$ is a Hopf ideal, since
$I(T)$ and $R^+$ are Hopf ideals of $T$, respectively of $R$.
Therefore, $T/(I(T) + R^+ T)$ is a commutative affine Hopf $k$-algebra
with a unique maximal ideal, namely $T^+ / (I(T) + R^+ T).$ By the
Nullstellensatz, this maximal ideal must be nilpotent. But, since
$k$ has characteristic 0, $T/(I(T) + R^+T)$ is semiprime by
Cartier's theorem \cite[Theorem 11.4]{Wa}, so the reverse
inclusion for (\ref{counitepi}) is proved.

Now, applying winding automorphisms of $T$ to (\ref{counitepi}),
bearing in mind that these preserve the Hopf subalgebra $R$ of $T$,
we deduce that, for all $M \in \Xi (T)$, setting $\mathfrak{m} = M \cap R$,
\begin{equation}\label{globalepi}
M = I(T) + \mathfrak{m}T.
\end{equation}
Since $T/M \cong k$, (\ref{globalepi}) can be restated as: for
every maximal ideal $\mathfrak{m}$ of $R$ with $I(T) \cap R
\subseteq \mathfrak{m}$,
\begin{equation} \label{final}
R + I(T) + \mathfrak{m}T = T.
\end{equation}

Set $\overline{R} := R/(I(T) \cap R)$, a commutative affine Hopf
algebra, which is therefore semiprime by a second application of
\cite[Theorem 11.4]{Wa}. To complete the proof of (ii) we show
first that $\overline{T} := T/I(T)$ is locally a finitely generated
$\overline{R}$-module. Let $Y$ be an indeterminate. There is an
algebra epimorphism $\psi$ from $\overline{R}[Y]$ onto $\overline{T}$,
sending the coset of $r + (I(T) \cap R)$ to $r + I(T)$ for $r \in R$,
and sending $Y$ to the coset of $X$ in $\overline{T}$. The kernel
of $\psi$ must be non-zero, since otherwise (i) would be contradicted.
Fix $\sum_{i=0}^n \bar{r}_i Y^i \in \mathrm{ker}\psi$, with
$\bar{r}_n \neq 0.$ Since $\overline{R}$ is semiprime, the
Nullstellensatz ensures that there is a maximal ideal $\mathfrak{m}$
of $\overline{R}$ with $\bar{r}_n \notin \mathfrak{m}.$ We therefore
deduce that $\overline{T}_{\mathfrak{m}}$ is a finitely generated
$\overline{R}_{\mathfrak{m}}$-module, namely
\begin{equation}\label{localfg}
\overline{T}_{\mathfrak{m}} = \sum_{i=0}^{n-1}
\overline{R}_{\mathfrak{m}}\bar{X}^i.
\end{equation}
Next, note that the winding automorphisms of $T$ preserve
$I(T) \cap R$, hence induce automorphisms of $\overline{R}$ which
permute transitively the maximal spectrum of that algebra. Thus,
applying these winding automorphisms to (\ref{localfg}), we see
that (\ref{localfg}) is actually valid for \emph{every} maximal
ideal of $\overline{R}$ (and in fact, although we shall not need
this, even with the same module generating set
$\{\bar{1}, \bar{X}, \ldots , \bar{X}^{n-1}\}$, because of the
nature of $\Delta (X)$ guaranteed by Theorem \ref{permit}(i)(b)).

Suppose now that
\begin{equation}\label{west}
I(T) + R \subsetneq T.
\end{equation}
Thus there is a maximal ideal $\mathfrak{n}$ of $\overline{R}$
such that $(T/(I(T) + R))_{\mathfrak{n}} \neq 0.$ But this
$\overline{R}_{\mathfrak{n}}$-module is just
$(\overline{T}/\psi(\overline{R})_{\mathfrak{n}}$, which is
finitely generated by (\ref{localfg}). As such, it has a simple
factor which is isomorphic to
$\overline{R}_{\mathfrak{n}}/\mathfrak{n}\overline{R}_{\mathfrak{n}}.$
However, this forces
$$ R + I(T) + \mathfrak{n}T \subsetneq T, $$
contradicting (\ref{final}). This completes the proof of (ii).
\end{proof}

\subsection{Proof of Theorem \ref{varHOEs}}
By Lemma \ref{varHOEs}(ii) there exists $d \in R$ such that
$x - d \in I(T).$ Thus, for all $r \in R$,
$$ I(T) \ni (x - d)r - \sigma(r)(x - d) = \delta (r) - dr
+ \sigma (r)d. $$
But this element also belongs to $R$, and $I(T) \cap R = 0$
by Lemma \ref{varHOEs}(i), so that
$$ \delta (r) = dr - \sigma (r)d; $$
in other words, $\delta$ is an inner $\sigma$-derivation of $R$.
Therefore, if we set $\tilde{x} := x - d$, then
$T = R[\tilde{x}; \sigma]$, by \cite[Exercise 2Y]{GW}.

(ii) Since $R \otimes R$ is a domain, Theorem \ref{permit} applies,
so that $\sigma$ is the winding automorphism given by a central
element of $G$, and, noting that $\tilde{x} \in I(T) \subseteq T^+$,
$\Delta (\tilde{x})$ has the form (\ref{copx}), for some group-like
element $a \in R$, and $w \in R \otimes R$. Write $w$ as
$w = w_1 \otimes w_2$, where the elements $\{w_1\}$ are linearly
independent over $k$.

From its definition, the ideal $I(T)$ is invariant under all algebra
automorphisms of $T$, and is generated by $\tilde{x}$. From the form
(\ref{copx}) for the coproduct of $\tilde{x}$ it follows that, for
each $\chi \in \Xi (T) = \Xi (R) = G$,
\begin{equation}\label{tau}
\tau^r_{\chi}(\tilde{x}) = \tilde{x} + r_{\chi} \in I(T),
\end{equation}
where $r_{\chi} = \sum w_1 \chi(w_2) \in R.$ Since $I(T) \cap R = 0$
we must have $r_{\chi} = 0$ for all $\chi \in G.$ However, $R$ is a
commutative affine domain by hypothesis, so the nullstellensatz
guarantees that, if $w \neq 0$, there exists $\chi \in G$ such that
$\chi(w_2) \neq 0$ for some term $w_2$ in the expression for $w$.
Then the linear independence of $\{w_1\}$ ensures that $r_{\chi}
\neq 0$. This is a contradiction, so $w$ must be 0; that is,
$\tilde{x}$ is skew primitive, as required.

(iii) This is a special case of \cite[Theorem 1.3]{Pa}.

\begin{rem}
Theorem \ref{varHOEs}(i) does not remain true if the commutativity
hypothesis on $R$ is omitted: consider for example
$T = U(\mathfrak{sl}(2, \mathbb{C})$, whose structure as IHOE is
described in  Examples \ref{iterdefn}(iii).
\end{rem}

\subsection{IHOEs satisfying a polynomial identity}
\label{PI}
Recall that, over a field of characteristic 0, an enveloping algebra
of a Lie algebra satisfies a polynomial identity only if the Lie
algebra is commutative, \cite{La}, (or see \cite{Pas} for the
infinite dimensional case). Evidence that, at least in
characteristic 0, the algebra structure of IHOEs resembles that of
enveloping algebras, is given by the following theorem.

\begin{thrm} If an IHOE $H$ over the field $k$ of characteristic 0
satisfies a polynomial identity, then $H$ is commutative.
\end{thrm}

\begin{proof} We induct on $n := \mathrm{GKdim} H,$ the result being
trivial for $n = 0.$

Let $\mathrm{GKdim}H = n$, and suppose the result is known for IHOEs
of smaller dimension. This forces $H$ to have the form
$H = R[x; \sigma, \delta]$ for a commutative polynomial Hopf
$k$-algebra $R$ in $n-1$ variables, the number of variables being
given by Theorem \ref{firstprop}(ii). By \cite{Jo},
\begin{equation}\label{Jon}
\mathrm{PIdegree}( H) = \mathrm{PIdegree}(R[x; \sigma]),
\end{equation}
and then it follows from \cite{DS} that $\sigma$ has finite order.
However, $\sigma$ is a winding automorphism of $R$, by Theorem
\ref{permit}(i)(d). By $\S$\ref{maxIHOE}, the character group $U$
of $H$ is unipotent. In particular, since $k$ has characteristic 0,
$U$ is torsion free. Therefore $\sigma$ must be the identity map,
and we deduce that
\begin{equation}\label{clinch}
\mathrm{PIdegree}(R[x;\sigma]) = 1.
\end{equation}
The result follows from (\ref{Jon}) and (\ref{clinch}).
\end{proof}

\section*{Acknowledgments}
The authors thank Christian Lomp for pointing out an error
in a previous version of the paper.
S. O'Hagan was supported by by DTA funds of the UK EPSRC.
J.J. Zhang was supported by the US National
Science Foundation (NSF grant DMS-0855743 and DMS-1402863).
Brown and Zhang gratefully acknowledge the financial support and
excellent working conditions of the Banff International Research
Station, Canada, the Mathematical Sciences Research Institute,
Berkeley, USA, and the Institute of Mathematics of the Polish
Academy of Science, Bedlewo, Poland.


\end{document}